\documentclass[10pt]{article}
\usepackage{amsmath,amssymb,amsthm}
\usepackage{tikz-cd}

\usepackage[english]{babel}
\usepackage[totalwidth=160mm, totalheight=240mm]{geometry}

\usepackage[auth-sc, affil-it]{authblk}
\usepackage{tikz}
\usepackage{tikz-3dplot}
\usetikzlibrary{calc}
\usepackage{verbatim}

\usepackage{pstricks}
\usepackage{pgf}
\usepackage{curve2e}

\newcommand{\bbN}{\mathbb{N}}
\newcommand{\bbZ}{\mathbb{Z}}
\newcommand{\bbQ}{\mathbb{Q}}
\newcommand{\bbR}{\mathbb{R}}

\newcommand{\Pc}{\mathcal{P}}
\newcommand{\Qc}{\mathcal{Q}}

\newcommand{\Lc}{\mathcal{L}}

\newcommand{\Hc}{\mathcal{H}}
\newcommand{\Gc}{\mathcal{G}}
\newcommand{\Fc}{\mathcal{F}}

\newcommand{\Dc}{\mathcal{D}}
\newcommand{\Cc}{\mathcal{C}}

\newcommand{\Uc}{\mathcal{U}}

\theoremstyle{plain}

\newtheorem{main-theorem}{Theorem}

\newtheorem{theo}{Theorem}[section]

\newtheorem{lemm}[theo]{Lemma}

\newtheorem{deff}[theo]{Definition}
\newtheorem{exam}[theo]{Example}
\newtheorem{rema}[theo]{Remark}

\newtheorem*{clai-nn}{Claim}
\newtheorem*{theorem*}{Theorem}
\newtheorem*{remark*}{Remark}
\newtheorem*{question*}{Question}

\newtheorem*{schoencond*}{Sch\"onflies Condition}

\usepackage{CJK}

\begin{document}
\begin{CJK}{GBK}{kai}
\title{A Classification of Fractal Squares}
\date{}
%\date{}
 \author[1]{Gregory Conner}
\author[2]{Curtis Kent}
\author[3]{Jun Luo}
\author[4]{Yi Yang}

%\thanks{Supported by Chinese National Natural Science Foundation Projects \# 11771391 and 11871483.}

\affil[1]{Department of Mathematics,    Brigham Young University, Provo, UT 84602, USA \newline conner@mathematics.byu.edu}

\affil[2]{Department of Mathematics,    Brigham Young University, Provo, UT 84602, USA \newline curtkent@mathematics.byu.edu}

\affil[3]{School of Mathematics,     Sun Yat-sen University, Guangzhou 510275, China \newline luojun3@mail.sysu.edu.cn}

\affil[4]{School of Mathematics (Zhuhai),     Sun Yat-sen University, Zhuhai 519082, China\newline yangy699@mail.sysu.edu.cn {\rm(corresponding author)} }
\maketitle
\begin{abstract}
Let  $\lambda_K:\bbR^2\rightarrow\{0,1,\ldots\}\cup\{\infty\}$ be the lambda function of a planar comapctum $K$, as defined in MR4488162. It is known that a planar continuum  is locally connected if and only if its lambda function vanishes everywhere, or equivalently, $\lambda_K(K)=\{0\}$. In this article we show  that every fractal square $K$ satisfies  $\lambda_K(K)\subset\{0,1\}$ and find criterions to classify when  $\lambda_K(K)$  equals $\{0\}$, $\{1\}$ or $\{0,1\}$.  Here for any integer  $N\ge2$ and  any set $\Dc=\left\{(i,j): 0\le i,j\le N-1\right\}$ with cardinality $\ge2$, if we set $K^{(0)}=[0,1]^2$ and 
$\displaystyle K^{(n)}=\left\{\frac{x+d}{N}: x\in K^{(n-1)}, d\in\Dc\right\}(n\ge1)$ then   $K=\bigcap_nK^{(n)}$ is called a fractal square. 
\end{abstract}

{\footnotesize \tableofcontents }
\newpage
\section{Introduction and Main Results}
%\cite{Caratheodory13-a} prime end theory \cite{Nadler92} continuum theory

Let $K$ be a planar compactum with locally connected components and without isolated points.  In three cases,  $K$ is well understood. First, if it is connected, $K$ is a Peano continuum; second,  if all components are points, $K$ is a Cantor set; third, $K$ is the product of $[0,1]$ and a Cantor set. Such  $K$ may arise as an invariant set  of certain dynamical system, like rational  Julia sets or  self-affine sets.  The first category includes Julia sets  of hyperbolic rational maps $f$. In such a case, every component of the Julia set is locally connected \cite[Corollary 1.3]{PilgrimTan00}; moreover, with the possible exception of finitely many periodic components $J_0$ satisfying $f^p(J_0)=J_0$ for some $p\ge1$ and their preimages, every other component of the Julia set is either a single point or a Jordan curve \cite[Theorem 1.2]{PilgrimTan00}.  
The second category includes fractal squares. By  \cite[Theorem 2]{LRX-2022}, every component of a fractal square is locally connected.  

An issue of interest is to classify such  $K$ from certain viewpoints.  We pick a topological point of view and use the lambda function $\lambda_K:\bbR^2\rightarrow\bbN\cup\{\infty\}$ introduced in \cite{FLY-2022}. See also Definition \ref{def:lambda}. Notice that,  if  $x\notin K$ then $\lambda_K(x)=0$. This means that a single point off $K$ is irrelevant to the topology of $K$. So the key is to investigate the set $\lambda_K(K)$. When $K$ is a continuum, $\lambda_K(K)=\{0\}$ if and only if $K$ is a Peano continuum. When $K$ is the product of $[0,1]$ with a Cantor set, $\lambda_K(K)=\{1\}$. When $K$ is disconnected, for $\lambda_K(K)$ to equal $\{0\}$ it is necessary and sufficient that every component of $K$ is a Peano continuum and for any constant $C>0$ at most finitely many components of $K$ are of diameter  $>C$. To illustrate how large the set $\lambda_K(K)$ can become, we refer to Examples \ref{key_example_2} and \ref{key_example} for the contruction of  simple planar compacta $K$ with locally connected components and without isolated points such that $\lambda_K(K)=\{0,1,\ldots,n\}$ for some $n\ge2$. 

The main goal of this paper is to show that if $K$ is a fractal square $\lambda_K(K)$ equals $\{0\}$, $\{1\}$, or $\{0,1\}$. On the one hand, each of these cases is possible. On the other, we can find criterions to classify which case happens for a concrete fractal square. For the moment, hyperbolic rational Julia sets are untouched hereafter, although we think that  similar results still hold for them. 

A fractal square may be defined as follows.
Given $N\ge2$ and $\Dc\subset\{0,1,\ldots,N-1\}^2$ with  cardinality $\#\Dc\ge2$, we obtain a finite family  $\Fc_\Dc=\Big\{f_d: d\in\Dc\Big\}$, where each $\displaystyle f_d(x)=\frac{x+d}{N}$ is a contraction of ration $\frac1N$. Following Hutchison \cite{Hutchinson81}, we call $\Fc_\Dc$  an {\bf iterated function system} (shortly,  IFS).
%with digit set $\Dc$ and contraction ratio $\frac1N$.
%The elements $d=(i,j)\in\Dc$ are often referred to as the {\em digits} of $K$. The translation $x\mapsto x+d$ and the dilation $d\mapsto Nd$ are defined in the usual way. 
It is known that there is a unique nonempty compactum $K=K(N,\Dc)$ satisfying the set equation $K=\bigcup_{d\in\Dc}f_d(K)$. See \cite[p.124, Theorem 9.1]{Falconer90} or \cite[$\S3.1$]{Hutchinson81}.

\begin{deff}\label{def:FS}
We call $K=K(N,\Dc)$ a {\bf fractal square} of order $N$ and $\Dc$ the {\bf digit set}. 
\end{deff}
%When $N$ is clear from the context, we also write $K_\Dc$ instead of $K(N,\Dc)$. 
%\footnote{\tb{We save $N$ as the only notation for the order of a fractal square.}}

A fractal square is a  self-similar set satisfying the open set condition \cite[pp.128-129]{Falconer90}. 
 There are many interesting results about fractal squares  from the literature. See \cite{DaiLuoRuanWang_23, Lau-Luo-Rao13, LRX-2022, Luo-Liu16, Rao-Wang-Wen17,Ruan-Wang17,RuanWangXiao_23,Xiao21} for recent articles that focus on topological properties. See also the survey \cite{Akiyama-Thuswaldner04} and the references therein.

Given a  fractal square $K$,   all its components are locally connected. Such a component may be a point, a line segment, a dendrite which then has a trivial  fundamental group, a Peano conitnuum with non-trivial fundamental group or one with no local cut point, which is then homeomorphic to Sierp\'nski's universal plane curve \cite[Theorem 3]{Whyburn58}.
 The diversities demonstrated in the topology of $K$  have made the study of fractal squares  well worth the effort.

On the other hand, the topology of $K$ also appears to be restricted in one way or another. For instance, by the discussions in \cite{Xiao21} one can infer two simple dichotomies. The first one states that $K$  has either finitely or uncountably many components. The second says that   $K$ has either  none or uncountably many point components.  Actually, if $K$ has both point components and non-degenerate ones,  the non-degenerate components consitute a set, say $K_c$, whose Hausdorff dimension $\dim_HK_c$ is strictly smaller than that  of $K$ \cite[Theorem 1.1]{Huang-Rao21}.

The topology of fractal squares is also interesting from a dynamic point of view, since fractal squares $K$ are often related to the action on $\bbR^2$ of certain transformation semi-groups. This viewpoint was taken in \cite{Lau-Luo-Rao13}, though the authors did not make explicit statements concerning such dynamical systems. To illustrate that, we address the necessary settings below.

Given  $K=K(N,\Dc)$, one may consider  $\bbZ^2$ as a group of translations $x\mapsto x+d$ acting on $\bbR^2$, with $d\in\bbZ^2$,  and look at the semi-group $\mathcal{G}_N$  generated by  group $\bbZ^2$ and the homothety $\psi_N(x)=Nx$. The action of this semi-group on $\bbR^2$ gives rise to a discrete dynamical system.  Thus a set $X\subset\bbR^2$ is said to be {\em invariant under $\mathcal{G}_N$} if $g(X)\subset X$ for all $g\in\mathcal{G}_N$. In particular,  $H=K+\bbZ^2=\{x+d: x\in K, d\in\bbZ^2\}$  is invariant under $\mathcal{G}_N$, since $H=H+\bbZ^2$ and $\psi_N(H)\subset H$.
In what follows, we set $K^{(0)}=[0,1]^2$. Following \cite{Lau-Luo-Rao13},  for $n\ge1$ we further set
\begin{equation}\label{eq:approx}
\displaystyle K^{(n)}=\bigcup_{d\in\Dc}\frac{K^{(n-1)}+d}{N}.
\end{equation}
Then $K^{(n)}=\bigcup_{d\in\Dc_n}\frac{[0,1]^2+d}{N^n}$, where $\Dc_1=\Dc$ and 
\begin{equation}\label{eq:D_n}
\Dc_{n}=\Dc_1+N\Dc_{n-1}=\left\{\sum_{i=0}^{n-1}N^id_i:\ d_i\in\Dc\right\}.
\end{equation}

\begin{deff}\label{def:approx}
Call $K^{(n)}$ the {\bf $n$-th approximation} of $K=K(N,\Dc)$. 
\end{deff}

\begin{rema}
To emphasize  $N$ and $\Dc$,  we  often write $K^{(n)}(N,\Dc)$ instead of $K^{(n)}$ and $H(N,\Dc)$ instead of $H$.
As $\left\{K^{(n)}: n\ge1\right\}$ is a decreasing sequence with $K=\bigcap_nK^{(n)}$, one may reveal certain properties of $K$  by analyzing  $K^{(n)}$. For instance, if  the number $\beta_0\Big(K^{(2)}\Big)$ of components of $K^{(3)}$ equals that of $K^{(2)}$, then $\beta_0(K)=\beta_0\left(K^{(3)}\right)$. See Theorem \ref{theo:finite_tau_extended}. % \cite[Lemma 1.5 and Theorem 1.6]{Xiao21}. 
Here $\beta_0(\cdot)$ usually  denotes the number of path components. Since every component of a fractal square $K$ is locally connected (hence path connected), we are safe to denote by $\beta_0(K)$ the number of components of $K$.
\end{rema}

To facilitate our discussions, we introduce more terminology below.
Pick an infinite line $L$, let $\overline{{\rm Orb}(L)}$ be the closure of 
${\rm Orb}(L)=\left\{g(L): g\in\mathcal{G}_N\right\}$,
the orbit of $L$ under  $\mathcal{G}_N$.  Let 
\begin{equation}\label{eq:half-open}
\Dc(L)=\left\{d\in\{0,1,\ldots,N-1\}^2: \ 
\frac{[0,1)^2+d}{N}\ \bigcap\ \overline{{\rm Orb}(L)}\ne\emptyset\right\}.
\end{equation}
Instead of using just one line $L$, one may start from a closed set $\Lc$ which consists of two or more or even infinitely many lines, of the same slope or of different slopes, and identify the set $\Dc(\Lc)$ in a similar way. %Let $\frak{L}(\bbR^2)$ consist of nonempty sets of infinite lines in $\bbR^2$. 
Then $\Lc\mapsto\Dc(\Lc)$ is  well-defined and satisfies $\Dc(\Lc)\subset\{0,1,\ldots,N-1\}^2$. 

\begin{deff}\label{def:digit_operator}
Call $\Lc\mapsto\Dc(\Lc)$ the {\bf digit operator}.
\end{deff}

 For any closed set $X$ invariant under $\mathcal{G}_N$, let $\frak{L}(X)$ be the union of  all the infinite lines lying in $X$.  
Then $\Dc\mapsto\frak{L}\Big(H(N,\Dc)\Big)$ is well-defined. So is the composite  
$\Lc\mapsto\frak{L}\Big(K(N,\Dc(\Lc))\Big)$, whose output  is a closed set invariant under $\mathcal{G}_N$ containing $\overline{{\rm Orb}(\Lc)}$. Moreover, for a totally disconnected fractal square $K(N,\Dc)$ there is no set $\Lc$ (consisting of infinite lines) with 
$\Dc=\Dc(\Lc)$. There are other  $K(N,\Dc)$, which are not totally disconnected, for which no $\Lc$ with $\Dc=\Dc(\Lc)$ can be found. See $K(5,\Dc_0)$ in Example \ref{exmp:slope_1} for such fractal square. In Examples \ref{exmp:no_line_dendrite} and \ref{exmp:no_line} we give two fractal squares, one is a dendrite and the other is not. Neither of them contains a line segment. Therefore, their digit sets cannot be obtained by using the digit operator $\Lc\mapsto\Dc(\Lc)$.
Here the containment  $\frak{L}\Big(H(N,\Dc(\Lc))\Big)\supset \overline{{\rm Orb}(\Lc)}$ may be proper. This is the case, if we set $N=3$ and $\Lc=L_1\cup L_2$, where $L_1:x_2=\frac{x_1}{2}$ and $L_2:x_2=\frac23$, since $\Dc(\Lc)=\{0,1,2\}^2\setminus\{(1,1)\}$ and $\frak{L}\Big(H(N,\Dc(\Lc))\Big)$ contains all the vertical (or horizontal) lines that cross $(\omega+i,0)\ \text{or}\ (0,\omega+i)$ with $i\in\bbZ$ and $\omega$ belonging to Cantor's ternary set.

\begin{exam}\label{exmp:slope_1}
Fix $N=5$. Let $L_1: x_2=x_1$, $L_2: x_2=x_1-\frac15$ and $L_3: x_2=x_1+\frac15$ be infinite lines of slope $1$ with intercepts $0,-\frac15$ and $\frac15$. Setting $\Dc_1=\Dc(L_1)$, $\Dc_3=\Dc(L_1\cup L_2)$, $\Dc_3=\Dc(L_1\cup L_2\cup L_3)$ and $\Dc_0=\Dc_2\setminus\{(0,4)\}$.  Figure \ref{PC_or_NOT} illustrates $K^{(n)}(5,\Dc_i)(n=2,4)$ for $0\le i\le3$. 
\begin{figure}[ht] \begin{center} 
\begin{tabular}{cccc}
\includegraphics[width=3.5cm]{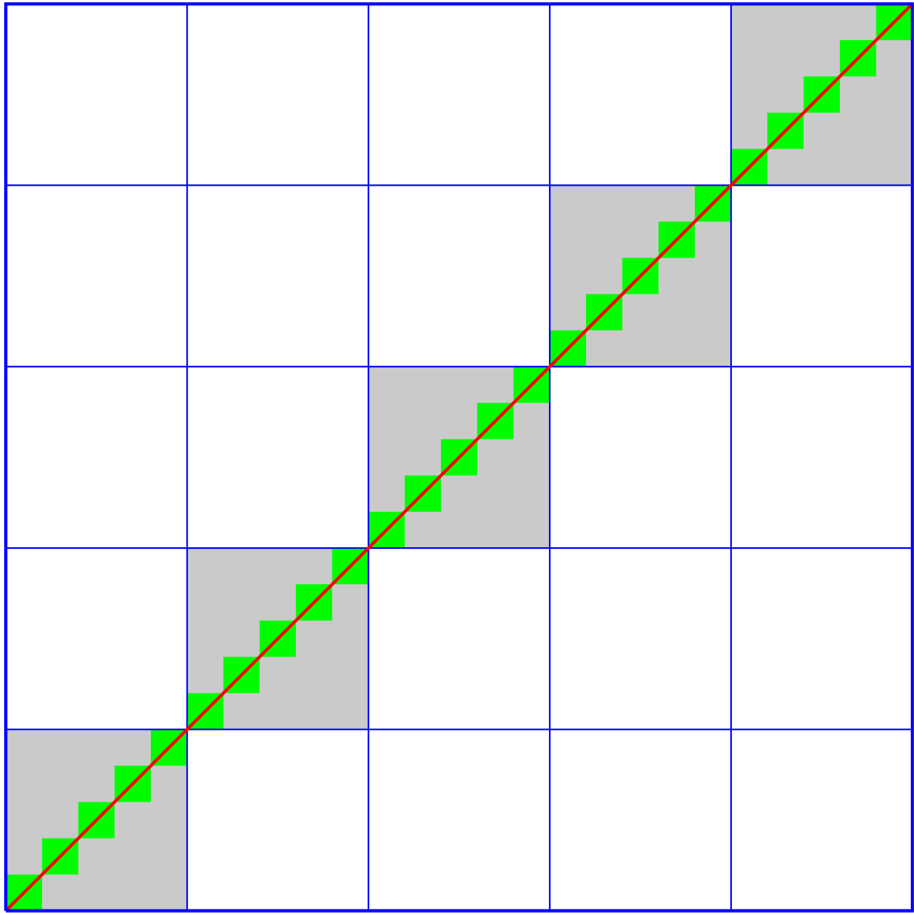} &
 \includegraphics[width=3.5cm]{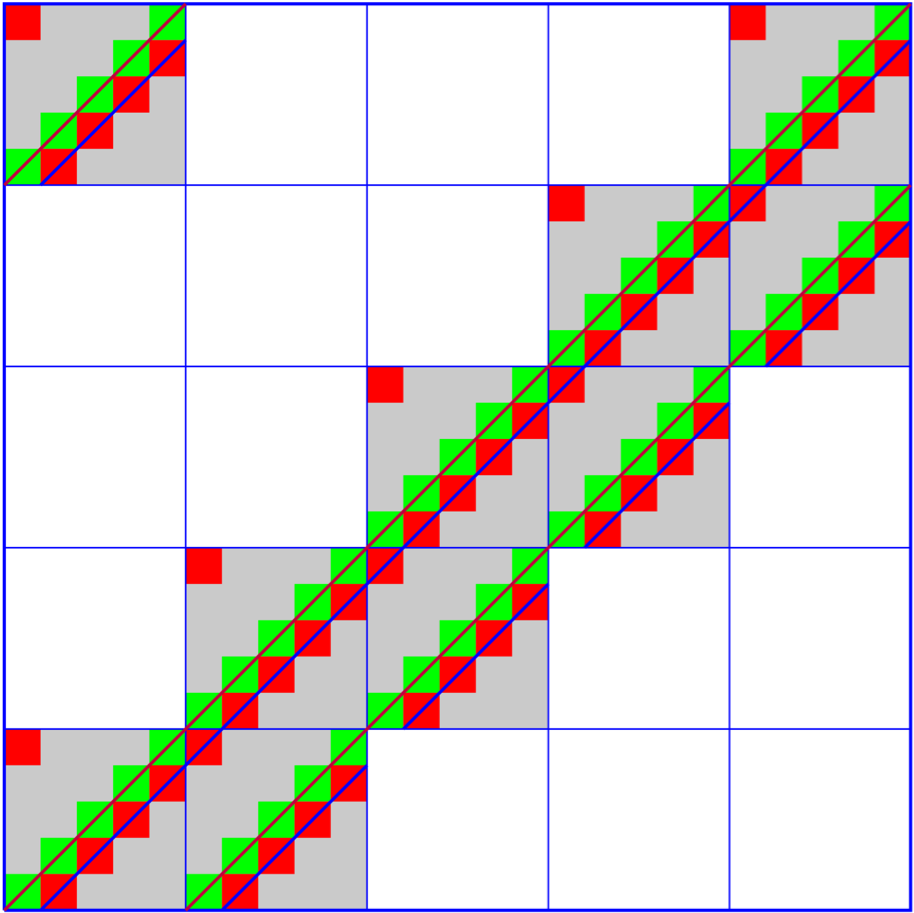}&
 \includegraphics[width=3.5cm]{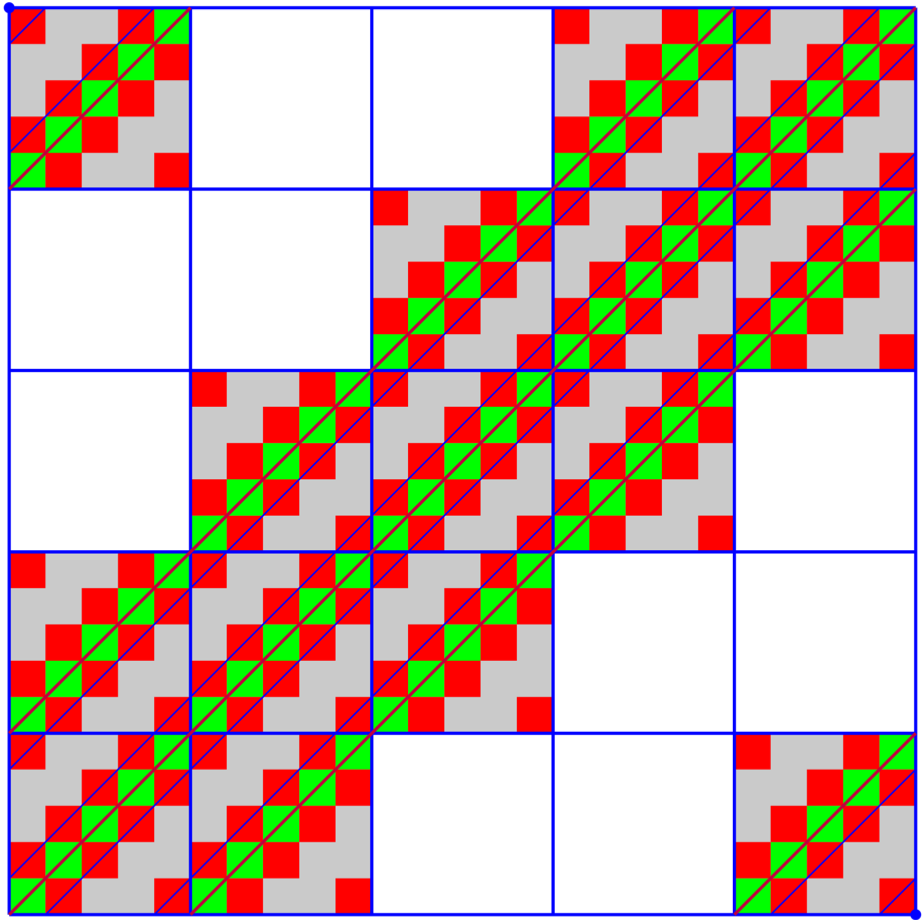} &
 \includegraphics[width=3.5cm]{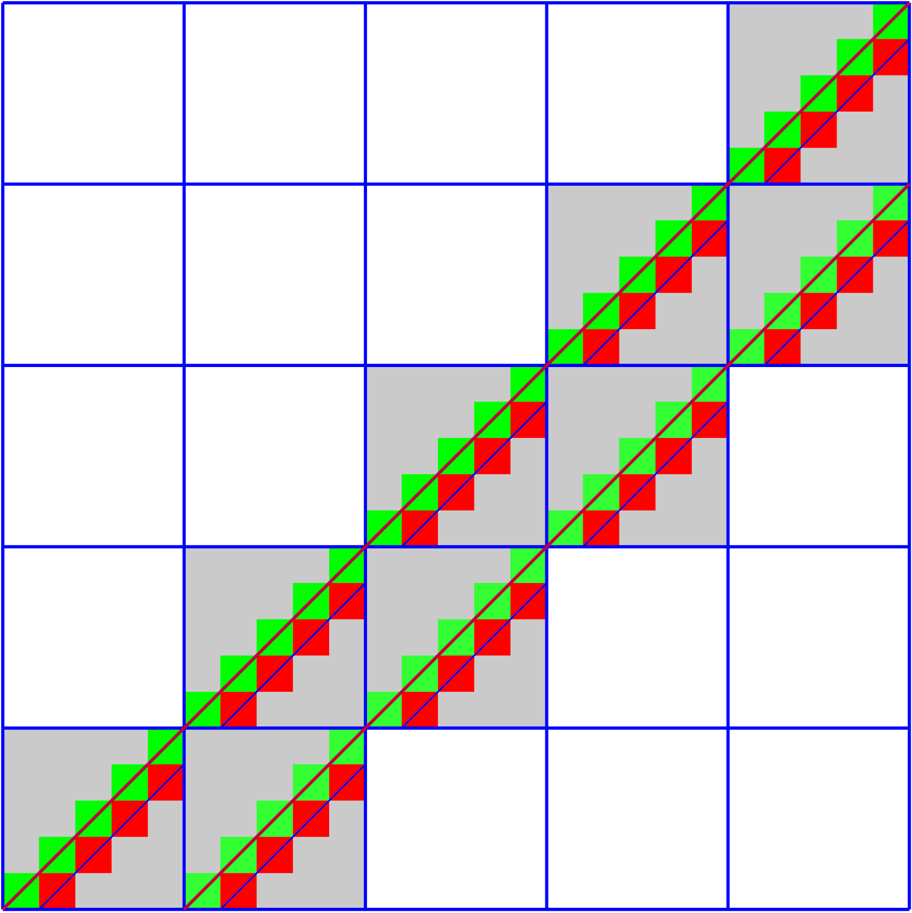}\\
 \includegraphics[width=3.5cm]{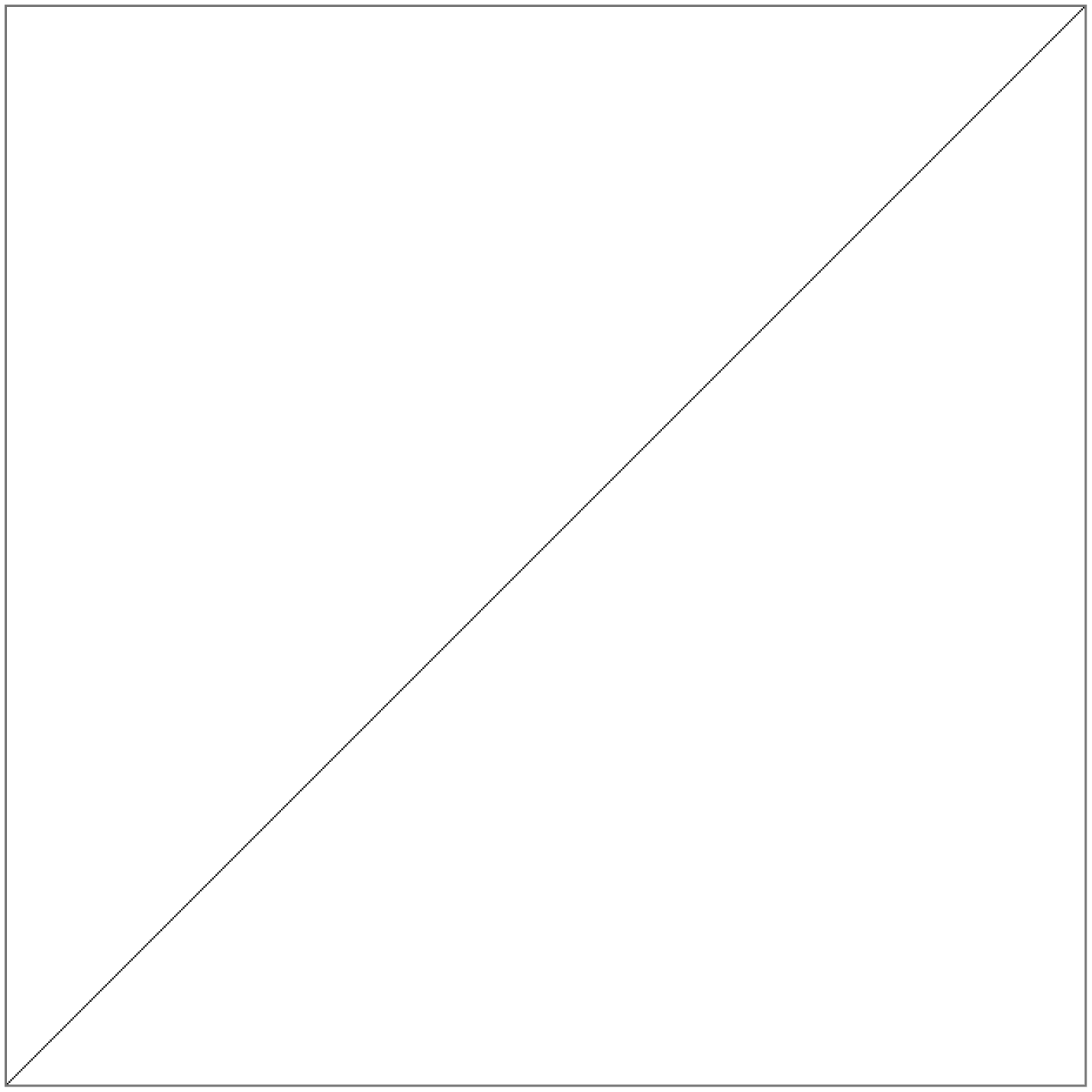}&
  \includegraphics[width=3.5cm]{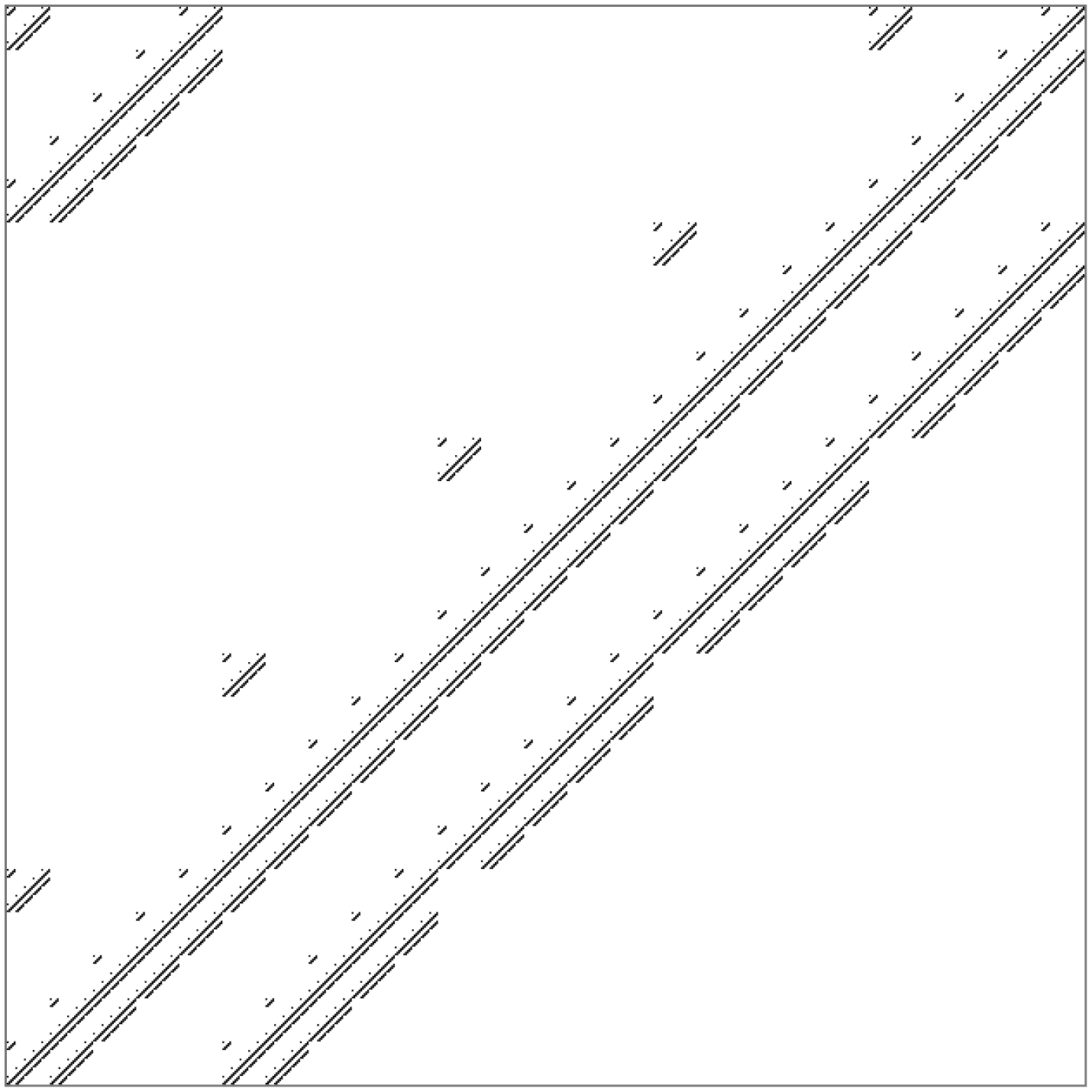}&
 \includegraphics[width=3.5cm]{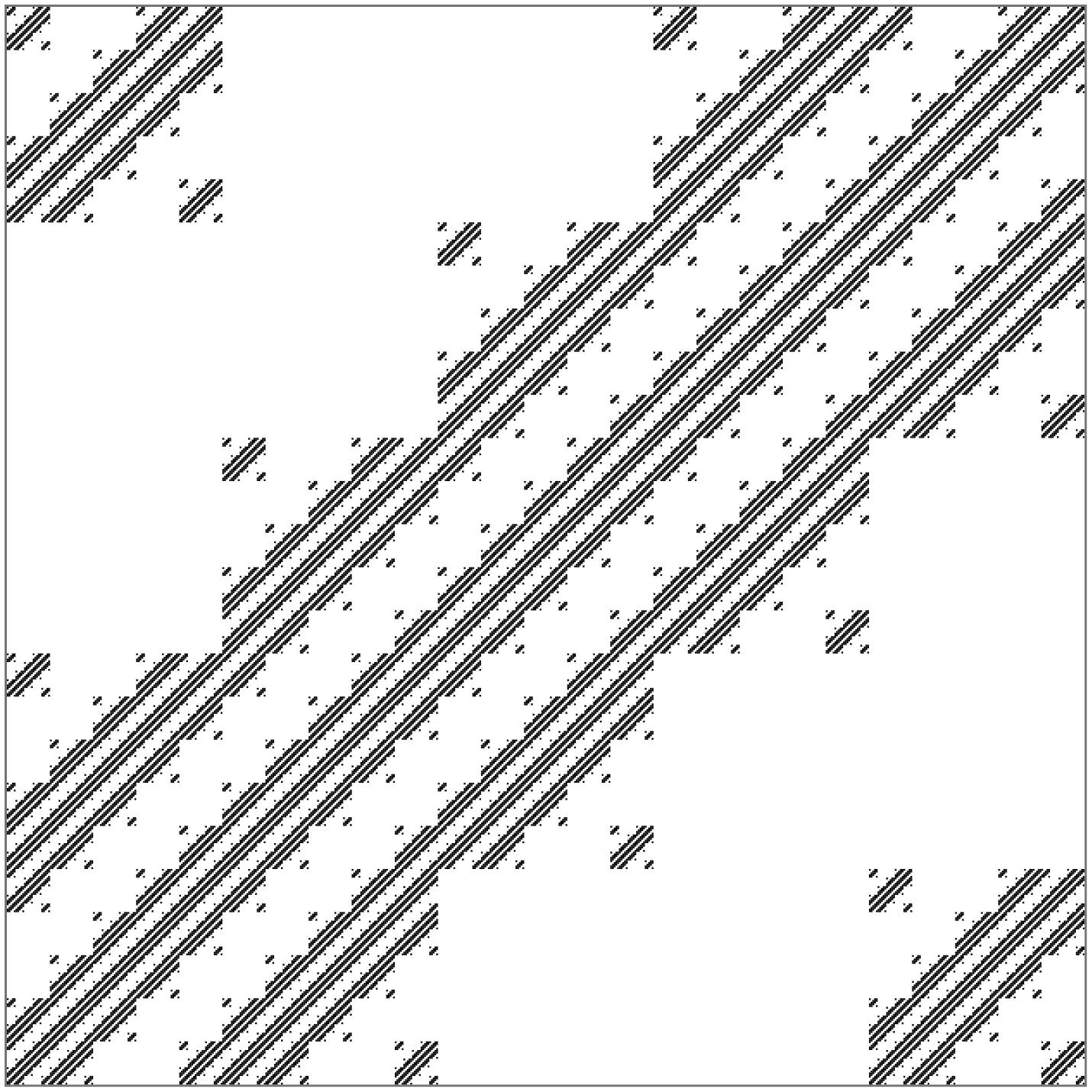} &
 \includegraphics[width=3.5cm]{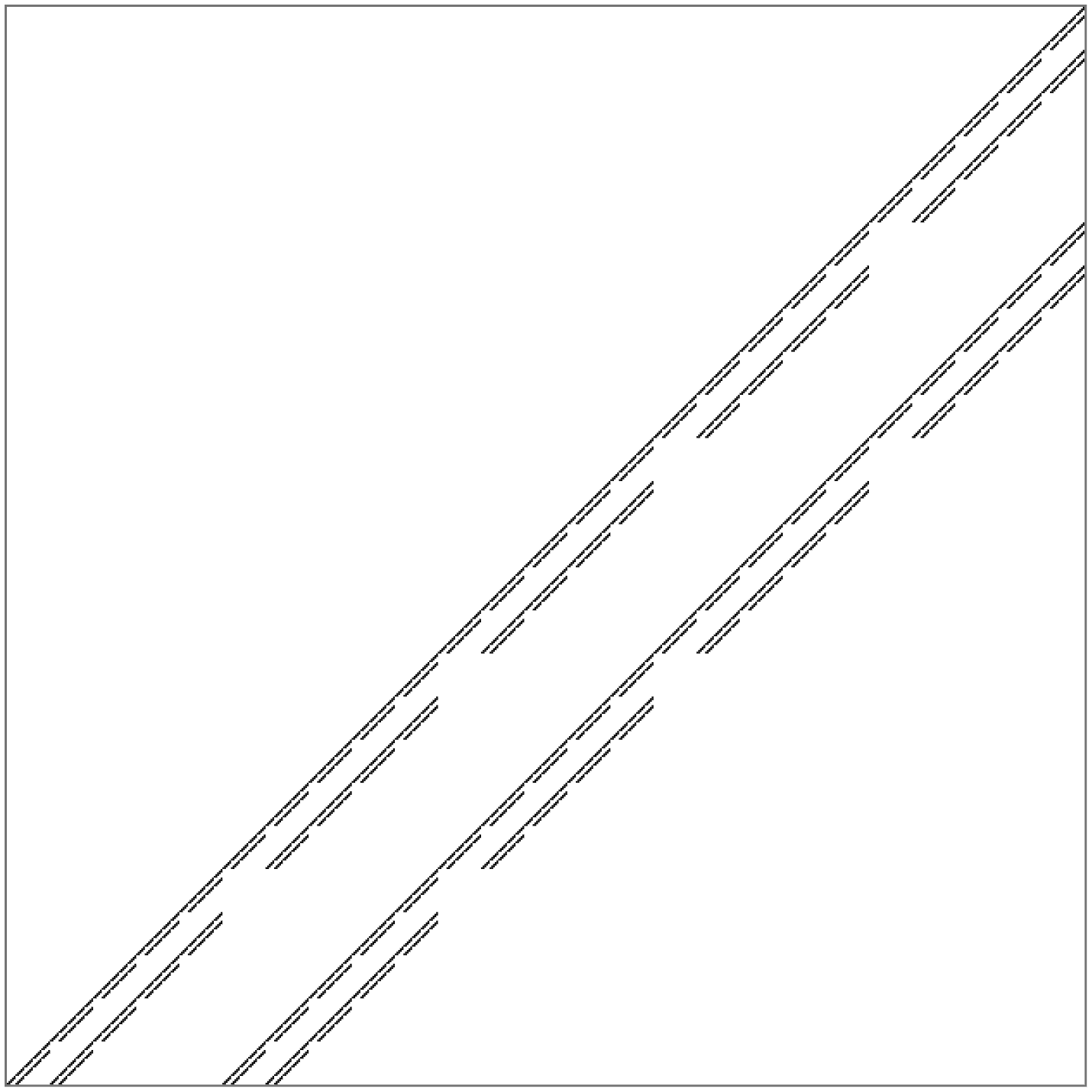} \\
$K^{(n)}(5,\Dc_1)(n=2,4)$&$K^{(n)}(5,\Dc_2)(n=2,4)$ & $K^{(n)}(5,\Dc_3)(n=2,4)$ & $K^{(n)}(5,\Dc_0)(n=2,4)$
 \end{tabular}
 \end{center} \vspace{-0.382cm} 
\caption{Approximations of Four Fractal Squares $K$.}\label{PC_or_NOT} \end{figure}
Then one can easily check the next observations.  First, $K(5,\Dc_1)$ is the diagonal of $[0,1]^2$. Second, $K(5,\Dc_2)$ has countably many line segments and infinitely many of them are of diameter $\ge\frac{4\sqrt{2}}{5}$. Third, $K(5,\Dc_3)$ has uncountably many line segments of diameter  
$\ge\frac{4\sqrt{2}}{5}$. Finally, $K(5,\Dc_0)$ has countably many line segments and they form a {\bf null sequence}, in the sense that for any $C>0$ at most finitely many of them are of diameter $\ge C$. 
%Moreover,  we shall have  \[L_1+\bbZ^2=\frak{L}\Big(H(5,\Dc_0)\Big)=\frak{L}\Big(H(5,\Dc_1)\Big)\subset\frak{L}\Big(H(5,\Dc_2)\Big)\subset \frak{L}\Big(H(5,\Dc_3)\Big).\] 
Notice that $\frak{L}\Big(H(5,\Dc_2)\Big)$ contains countably many lines and a line is the limit of an infinite sequence of lines $l_j\in\frak{L}\Big(H(5,\Dc_2)\Big)$ if and only if it is contained in $L_1+\bbZ^2$. Moreover, the  part of 
$\frak{L}\Big(H(5,\Dc_3)\Big)$ contained in the strip bounded by $L_1\cup L_1+(0,1)$ consists of all the lines with slope one that crosses $(0,t)$ with $t$ lying in a Cantor set $\Cc$. This Cantor set $\Cc$ is the attractor of $\{g_1,g_2\}$, where  $g_1(t)=\frac{t}{5}$ and $g_2(t)=\frac{t+4}{5}$. It follows that $\dim_H\frak{L}\Big(H(5,\Dc_3)\Big)=1+\frac{\ln 2}{\ln 5}$.
\end{exam}

In the current paper, we  consider fractal squares $K$ and discuss  topological properties that are related to the connectedness and local connectedness.   Let us recall some notions.

A {\bf Peano compactum} means a compactum with locally connected components such that for any constant $C>0$ at most finitely many of them have a  diameter  $\ge C$. For instance, the fractal square $K(5,\Dc_i)(i=0,1)$ in Example \ref{exmp:slope_1} are Peano compacta while $K(5,\Dc_i)(i=2,3)$ are not. See \cite[Theorem 3]{LLY-2019} for  the motivation of Peano compactum from the study of polynomial Julia sets. See also \cite{LY-2022} concerning how Peano compactum plays a role  in extending Carath\'eodory's Continuity Theorem (\cite[Theorem 3]{Arsove68-a} or \cite[p.18]{Pom92}) to infinitely connected domains.

Due to \cite[Theorem 7]{LLY-2019}, we know that every planar  compactum $K$ allows a special upper semi-continuous (shortly, usc) decomposition into sub-continua, denoted by $\Dc_K^{PC}$, such that the following two requirements are both satisfied. First, the resulting quotient space is a Peano comapctum. Second, $\Dc_K^{PC}$ refines every other such decomposition satisfing the first requirement.  

\begin{deff}\label{def:CD}
We call $\Dc_K^{PC}$ the {\bf core decomposition} of $K$ with Peano quotient.
\end{deff}

Notice that $\Dc_K^{PC}$ is also an atomic decomposition of $K$, in the sense of  \cite[Definition 1.1]{FitzGerald-Swingle67}.
% with respect to the property of being an usc decomposition into sub-continua that has a Peano compactum as its quotient space.
So, every element of $\Dc_K^{PC}$ is called an {\bf atom} (of order one) or an {\bf order one atom} of $K$.
Moreover, if $\delta$ is an order one atom of $K$, every atom of $\delta$ is called an {\bf order two atom} of $K$. Inductively we may define  {\bf atoms of order $m$} for $m\ge3$ and further define $\lambda_K: \bbR^2\rightarrow\bbN\cup\{\infty\}$ as follows.

Set $\lambda_K(x)=0$ for all $x\notin K$. Set $\lambda_K(x)=m-1$ for $x\in K$ such that there is a minimal $m\ge1$ with $\{x\}$  an order-$m$ atom of $K$ (thus $\{x\}$ is not an atom of order $m-1$, if $m>1$). Set $\lambda_K(x)=\infty$ for $x\in K$, if such an integer $m$ does not exist.

\begin{deff}\label{def:lambda}
We call $\lambda_K: \bbR^2\rightarrow\bbN\cup\{\infty\}$ the {\bf lambda function} of $K$.
\end{deff}

The lambda function $\lambda_K$ and its range $\lambda_K(K)=\{\lambda_K(x): x\in K\}$ quantify certain aspects of the topology of $K$. In particular, a continuum $K\subset\bbR^2$ is locally connected at $x\in K$ provided that $\lambda_K(x)=0$.
The following example demonstrates that  the local connectedness of $K$ at $x_0$ does not imply  $\lambda_K(x_0)=0$. 

\begin{exam}\label{exmp:LC_lambda=0}
Let $K$ be the union of $L_0=\{(t,0):0\le t\le1\}$ and the line segments $L_n(n\ge1)$ joining $(0,0)$ to $(1,\frac1n)$.  Then $K$ is the {\bf Closed Infinite Broom} \cite[p.139]{SS-1995} and $\lambda_K(x)=1$ if and only if $x\in L_0$. In particular, $K$ is locally connected at $x_0=(0,0)$ while $\lambda_K(x_0)=1$.\end{exam}

Given a continuum $K\subset\bbR^2$,   $\lambda_K(x)$ for specific $x$ measures the extent to which is $K$ not locally connected at $x_0$. Therefore, 
$\lambda_K(x)$ provides a numerical scale for non-local connectedness. A similar scale, based on fibers of planar continua, has been discussed in \cite{JLL16}.

With the aid of $\lambda_K$, one may ``quantify'' the classical Torhorst Theorem (\cite[p.106, (2.2)]{Whyburn42}), which states that  every complementary domain $U$ of a locally connected continuum $M\subset\bbR^2$ is bounded by a locally connected continuum.
We may strengthen this result as follows.
\begin{theorem*}[{\bf \cite[Theorem 2]{FLY-2022}}]
Given a compactum $K\subset\bbR^2$, a component $U$ of \ $\bbR^2\setminus\!K$ and  a nonempty compact set  $L\subset\partial U$, we have $\lambda_L(x)\le\lambda_{\partial U}(x)\le\lambda_K(x)$ for all $x\in\bbR^2$.
\end{theorem*}

The lambda function $\lambda_K$ will be of particular interest, if  $K$ arises from  complex dynamics (like polynomial Julia sets) or fractal geometry (like self-similar sets). In this paper, we concentrate on fractal squares $K$, which are special self-similar sets. First, we obtain the following.

\begin{main-theorem}[{\bf Dichotomy Theorem}]\label{maintheo:LLR_improved}
Let $K$ be  a fractal square and $H=K+\bbZ^2$. Then  one of the next assertions holds: (1) the components of $\bbR^2\setminus H$ are unbounded and every component of $K$ is either a single point or a line segment; (2) the components of $\bbR^2\setminus H$ have a diameter $\le\sqrt{2}(N^2+1)^2/N$ and $K$ is a Peano compactum. Consequently, we have $\lambda_K(K)\subset\{0,1\}$.
\end{main-theorem}

By \cite[Theorem 2]{LRX-2022}, every component of a fractal square is locally connected. This, however, does not imply that $\lambda_K(K)\subset\{0,1\}$. Examples \ref{key_example_2} and \ref{key_example} give   compacta $L_n(n\ge2)$  having locally connected components and satisfying $\lambda_{L_n}(L_n)=\{0,1,\ldots,n\}$. By  Theorem \ref{maintheo:LLR_improved},  if $\lambda_K(x)>0$ for some $x\in K$ then  every component of $K$ is either a point or a line segment.  From this it follows that $\lambda_K(K)\subset\{0,1\}$. Therefore, Theorem \ref{maintheo:LLR_improved} improves the following.
\begin{theorem*}[{\bf \cite[Theorem 2.2]{Lau-Luo-Rao13}}]
Let $K$ be  a fractal square and $H=K+\bbZ^2$. Then  one of the next assertions holds: (1) the components of $\bbR^2\setminus H$ are unbounded and every component of $K$ is either a single point or a line segment; (2) the components of $\bbR^2\setminus H$ have a diameter $\le\sqrt{2}(N^2+1)^2/N$.\end{theorem*}

Second, we characterize the fractal squares $K=K(N,\Dc)$ that satisfy $\lambda_K(K)=\{1\}$. Here it is necessary that $N\ge3$, since $K(2,\Dc)$ is connected for all $\Dc\subset\{0,1\}^2$.

\begin{main-theorem}\label{maintheo:product_form}
Given  $K=K(N,\Dc)$ with $N\ge3$, $\lambda_K(K)=\{1\}$ if and only if $K$ is the  product of  $[0,1]$ and a linear Cantor set that is the attractor of an IFS of the form $\left\{f_j(t)=\frac{s+t_j}{N}: 1\le j\le k\right\}$, where $2\le k\le N-2$ and $t_1,\ldots,t_k\in\left\{\frac{i}{N}: 0\le i\le N-1\right\}$ are distinct.
\end{main-theorem}

Finally, we characterize the fractal squares $K$ with $\lambda_K(K)=\{0,1\}$. Due to Theorem \ref{maintheo:LLR_improved}, we only need to consider the case that  $K$ has both points and line segments as components. All those line segments then have the same slope $\tau=\frac{r}{s}$, where $r$ and $s$ are co-prime integers. Let $(x_1,x_2)\xrightarrow{\Pi_\tau} x_2-\tau x_1$  be the projection of $\bbR^2$ onto $\bbR$. 
The components of $\prod_\tau\left(H_1\right)$ are open intervals, with diameters bounded away from zero and $\Omega_1=\left[0,\frac1s\right]\setminus\prod_\tau\left(\mathbb{R}^2\setminus H_1\right)$
 is a nonempty compact set  having finitely many components, which are isolated points or non-degenerate closed intervals. Assume with no loss of generality that $\Omega_1$ has $m=m_\Dc$ (or none) non-degenerate components and $q=q_\Dc$ (or none) isolated points. Then, we have the following.
\begin{main-theorem}\label{maintheo:criterion}
If $m\ge2$ then the Huasdorff dimension $\dim_H\lambda_K^{-1}(1)$ is strictly larger than one. If $m=1$ and $q\ge1$ then $\dim_H\lambda_K^{-1}(1)=1$. If $m\le1$ and $mq=0$ then $\lambda_K^{-1}(1)=\emptyset$.
\end{main-theorem}
%a point $t\in\Pi_tau(K_c)$ is isolated  if and only if $\lambda_K(x)=0$ for all $x\in \Pi_\tau^{-1}(t)\cap K$.

\begin{exam}\label{exmp:lambda_criterion}	
Let the digit sets $\Dc_i(i=1,2,3)$ be given as in Example \ref{exmp:slope_1}. Then the pair $(m,q)$ respectively equals $(0,1)$, $(1,1)$, and $(2,0)$. By Theorem \ref{maintheo:criterion}, the fractal squares $K_i=K(5,\Dc_i)$ respectively satisfy $\lambda_{K_1}^{-1}(1)=\emptyset$, $\dim_H\lambda_{K_2}^{-1}(1)=1$ and $\dim_H\lambda_{K_3}^{-1}(1)>1$.
\end{exam}

\begin{rema}\label{rem:check_lambda_1}
The integers $q_\Dc$ and $m_\Dc$ may be determined for any $K=K(N,\Dc)$ for which $N$ and $\Dc\subset\{0,\ldots,N-1\}^2$ are given.
So Theorem  \ref{maintheo:criterion}  characterizes all fractal squares that have both points and line segments as components. See Theorems \ref{theo:Omega} and \ref{theo:intercept_set} for further details.
\end{rema}

The rest of this paper is organized as follows. In section \ref{K_3} we prepare necessary notions and basic result, trying to illustrate how certain topological properties of $K$ are connected to $K^{(n)}$. In section \ref{proof_1} we prove Theorems \ref{maintheo:LLR_improved} and \ref{maintheo:product_form}. In section \ref{proof_3} we deal with Theorem \ref{maintheo:criterion}. In section \ref{examples_and_questions} we  give  compacta $L_n(n\ge1)$, concretely constructed, such that $\lambda_{L_n}(L_n)=\{0,\ldots,n\}$.
\section{Investigating $K$ by Analyzing the Approximations $K^{(n)}$}\label{K_3}

A fractal square of order $N=2$ allows three possibilities: (1) it is a line segment, (2) it is the whole unit square, (3) it differs from the classical Sierpinski Triangle by an affine bijection. Since we are mostly interested in topological properties, we will focus on $K=K(N,\Dc)$, of any given order $N\ge3$. Obviously, if $\mathcal{D}\subset\{0,1,\ldots,N-2\}^2$ then $K$ is  a Cantor set. So, if the opposite is not explicitly stated, we usually assume that there is $d=(i,j)\in\mathcal{D}$ with $\max\{i,j\}=N-1$.   

In this section, we consider $\Dc$ as the ``data'' of $K$ and address how the topology of $K$ is  limited by $\Dc$.
Actually, we will illustrate how the number of components $\beta_0(K)$ is determined by  $K^{(3)}$ and under what conditions the fundamental group $\pi_1(K)$ (when $K$ is connected) is trivial. For general self-similar sets, one would not expect such connections.  

Let us start from some useful graphs. The first one is classical and has an origin in \cite{Hata85}.

\begin{deff}\label{def:Hata}
Given a fractal square $K=K(N,\Dc)$, let $G_1=G_1(\Dc)$ be the graph with vertex set $\left\{f_d(K): d\in\Dc\right\}$ such that two vertices $f_d(K),f_{d'}(K)$ are  incident if and only if $d\ne d'$ and  the corresponding sets  intersect. We call $G_1$  the {\bf first Hata graph} of $K(N,\Dc)$.
\end{deff}

In the same way $G_1$ is also defined for general self-similar sets. By \cite[Theorem 4.6]{Hata85},  $K$ is connected if and only if $G_1$ is. Further set  $f_d=f_{d_1}\circ f_{d_2}\circ\cdots\circ f_{d_n}$ for all $d=\sum_{i=0}^{n-1}N^{i-1}d_i$ with $d_i\in\Dc$. In a similar way, we can define the graphs $G_n=G_n(\Dc)$  with vertex set  $\left\{f_d(K): d\in\Dc_n\right\}$ such that  $f_d(K),f_{d'}(K)$  share an edge if and only if $d\ne d'$ and the corresponding sets intersect.
\begin{deff}\label{def:Hata_high}
Call $G_n=G_n(\Dc)$ the $n$-th Hata graph of $K=K(N,\Dc)$.\end{deff}

Note that  $K(N,\Dc)$ is also a fractal square of order $N^n(n\ge2)$, with digit set $\Dc_n$. Therefore, we have $K(N,\Dc)=K\left(N^n,\Dc_n\right)$. In the sequel, we   investigate the topology  of  $K$ and how $\beta_0(K)$ is related to $K^{(n)}$ with $n\le3$. Let 
$G_1^*\subset G_1$ be the subgraph,  with the same vertex set,  such that $G_1^*$ has an edge joining $f_d(K)$ to $f_{d'}(K)$   if and only if  $K^{(2)}+d$ intersects $K^{(2)}+d'$. Then we have the following.
%Let $G_j^{appr}$ be the graph with vertex set $\Dc_j$ such that two distinct vertices $d=d_1+nd_2+\cdots+n^{j-1}d_j$ and $d'=d_1'+nd_2'+\cdots+n^{j-1}d_j'$ are incident provided that $f_d([0,1]^2)\cap f_{d'}([0,1]^2)\ne\emptyset$. We call $G_j^{appr}$ the {\bf $j$-th approximation graph} of $\Fc_\Dc$.
%A similar result even holds for higher dimensional counterparts of fractal squares \cite[Theorem 1.12]{DaiLuoRuanWang_23}. 

\begin{theo}[{\cite[Theorem 2.2]{DaiLuoRuanWang_23}}]\label{theo:K2}
$G_1=G_1^*$. Consequently, if $\beta_0\left(K^{(3)}\right)=1$ then $\beta_0(K)=1$.
\end{theo}

Theorem \ref{theo:K2} is also implied by the following.
\begin{lemm}\label{lem:K2}
If  $\left(K^{(2)}+d_1\right)\cap \left(K^{(2)}+d_2\right)\ne\emptyset$ for $d_1\ne d_2$ belonging to $\Dc_1$, there exist $e_1,e_2\in\Dc_2$ such that $\left(\frac{K+e_1}{N^2}+d_1\right)\cap\left(\frac{K+e_2}{N^2}+d_2\right)\ne\emptyset$. In particular, we have $\left(K+d_1\right)\cap \left(K+d_2\right)\ne\emptyset$.
\end{lemm}

\begin{proof} If a corner of $[0,1]^2$ does not lie in $K$, it does not lie in $K^{(1)}$ hence not in $K^{(2)}$. Thus the result is transparent whenever $|d_1-d_2|=\sqrt{2}$. So we only need to consider the case $|d_1-d_2|=1$. With no loss of generality, we may assume that $d_1=(0,0)$ and $d_2=(1,0)$.

The assumption  $\left(K^{(2)}+d_1\right)\cap \left(K^{(2)}+d_2\right)\ne\emptyset$ implies the existence of  $e_1,e_2\in\Dc_2$ such that $\frac{[0,1]^2+e_1}{N^2}+d_1$ and $\frac{[0,1]^2+e_2}{N^2}+d_2$ have at least one common point $x_0$. It will suffice to verify:
\begin{equation}\label{eq:K2}
\left(\frac{K+e_1}{N^2}+d_1\right)\cap\left(\frac{K+e_2}{N^2}+d_2\right)\ne\emptyset.	
\end{equation}
To do that, we pick  $e_i',e_i''\in\Dc_1$ for $i=1,2$ such that $e_i=e_i'+Ne_i''$ and set $Q_i=\frac{K^{(1)}+e_i}{N}$. Then there are three possibilities. First,   $e_1''=(N-1,k)$ and $e_2''=(0,k+1)$ for some $0\le k\le N-2$ and $x_0\in Q_1\cap Q_2$. See the left part of Figure  \ref{fig:H1*}.
\begin{figure}[ht]
\begin{center}
\includegraphics[width=13.5cm]{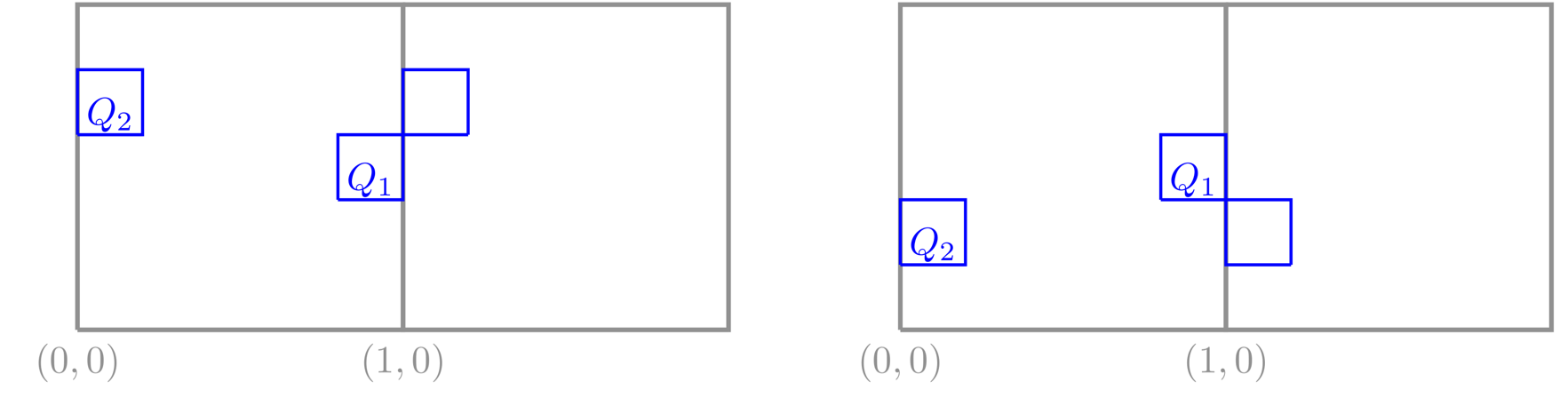}
\end{center}\vspace{-0.75cm}
\caption{Relative locations of $Q_1$ and $Q_2$ for the first and the second cases.}\label{fig:H1*}
\end{figure}
Second,    $e_1''=(N-1,k+1)$ and $e_2''=(0,k)\in\Dc$ for some $0\le k\le N-2$ and $x_0\in Q_1\cap Q_2$. See the right part of Figure \ref{fig:H1*}. Third,   $e_1''=(N-1,k)$ and $e_2''=(0,k)$ for some $0\le k\le N-1$ and $x_0\in Q_1\cap Q_2$.
In the first case, we have  $\{(0,0),(N-1,N-1)\}\subset\Dc$, which implies  $(0,0),(1,1)\in K$ and hence
\[
x_0\in\left(\frac{K+e_1}{N^2}+d_1\right)\bigcap\left(\frac{K+e_2}{N^2}+d_2\right).
\]
In the second, we have $\{(N-1,0),(0,N-1)\}\subset \Dc$, which implies $(1,0),(0,1)\in K$ and hence
\[
x_0\in\left(\frac{K+e_1}{N^2}+d_1\right)\bigcap\left(\frac{K+e_2}{N^2}+d_2\right).
\]
In the last, the point $u_1=\left(1,\frac{k}{N-1}\right)$ is fixed by $f_{e_1''}(x)=\frac{x+e_1''}{N}$ and coincides with $u_2+(1,0)$, where $u_2=\left(0,\frac{k}{N-1}\right)$ is fixed by $f_{e_2''}(x)=\frac{x+e_2''}{N}$. Thus \[
u_1\in\left(\frac{K+e_1}{N^2}+d_1\right)\bigcap\left(\frac{K+e_2}{N^2}+d_2\right).
\]
In all the cases, we have verified Equation \ref{eq:K2}. This completes our proof.
\end{proof}

To illustrate how the topology of $K^{(n)}$ may vary, when $n$ increases to $\infty$, we introduce the  following division rule. For any $i,j\ge1$,  let $K^{(i)}\bigotimes K^{(j)}$ be the set obtained by {\bf replacing every square $\frac{[0,1]^2+d}{N^i}$ with $d\in\Dc_i$ by $\frac{K^{(j)}+d}{N^i}$}. The resulting object is just $K^{(i+j)}$. 
\begin{deff}\label{def:K_operator}
Call  $K^{(i)}\bigotimes K^{(j)}$ the sub-division of $K^{(i)}$ by $K^{(j)}$.
\end{deff}
We consider $K^{(i)}\bigotimes K^{(j)}$ and $K^{(j)}\bigotimes K^{(i)}$ to be different operations, since $K^{(i)}\bigotimes K^{(j)}$ means we  start from $K^{(i)}$ and obtain $K^{(i+j)}$ by a sub-division rule based on $K^{(j)}$, while $K^{(j)}\bigotimes K^{(i)}$ means we  start from $K^{(j)}$ and obtain $K^{(i+j)}$ by a sub-division rule based on $K^{(i)}$.

\begin{theo}\label{theo:finite_tau_extended}
For any $n_0\ge2$, if $\beta_0\left(K^{(n_0+1)}\right)=\beta_0\left(K^{(n_0)}\right)$ then $\beta_0(K)=\beta_0\left(K^{(n_0)}\right)$.
\end{theo}
\begin{proof} Let $q=\beta_0\left(K^{(n_0)}\right)$. We will  inductively show that $\beta_0\left(K^{(n)}\right)=q$ for all $n\ge n_0$.
When $n\le n_0+1$, this is already given in the assumption. In the sequel we suppose that  $\beta_0\left(K^{(j)}\right)=q$ for  $n_0\le j\le n$ and infer that $\beta_0\left(K^{(n+1)}\right)=q$ still holds. Let $\Pc_1,\ldots, \Pc_q$ be the components of $K^{(n-1)}$ and $\Pc^\#_1,\ldots, \Pc^\#_q$ the components of $K^{(n)}$ with $\Pc^\#_j\subset\Pc_j$. Recall that $K^{(n)}=K^{(1)}\bigotimes K^{(n-1)}$ and $K^{(n+1)}=K^{(1)}\bigotimes K^{(n)}$.
For any $1\le j\le q$ and $d\in\Dc_1$, let $\Delta_{j,d}=\frac{\Pc_j+d}{N}$ and $\Delta^{\#}_{j,d}=\frac{\Pc^\#_j+d}{N}$. There are three observations. First,   all those sets $\Delta_{j,d}$ and $\Delta^{\#}_{j,d}$ are connected. Second, 
 the equations below are both valid:
\begin{align}\label{eq:rewrite}
K^{(n)}&= %\bigcup_{i=1}^q\ \ 
\bigcup_{j=1}^q\ \bigcup_{d\in\Dc_1}\ \Delta_{j,d} %\frac{\Pc_j+d}{N}
\\
K^{(n+1)}&=%\bigcup_{i=1}^q\ \ 
\bigcup_{j=1}^q\ \bigcup_{d\in\Dc_1}\ \Delta^{\#}_{j,d} %\frac{\Pc^\#_j+d}{N}
\end{align}
Finally, each $\Pc^\#_i$  consists of all $\Delta_{j,d}\in\Uc^{(i)}$.
For $1\le i\le q$, let $\Uc^{(i)}$ be the family of those $\Delta_{j,d}$ with $1\le j\le q$ and $d\in\Dc_1$ that are contained in $\Pc_i$. Similarly, let $\Uc^{\#(i)}$ be the family of those  $\Delta^\#_{j,d}$ with $1\le j\le q$ and $d\in\Dc_1$ that are contained in $\Pc^\#_i$. Then we define  two graphs  $\Gc^{(i)}$ and $\Gc^{\#(i)}$, respectively  with vertex sets $\Uc^{(i)}$ and $\Uc^{\#(i)}$, such that there is an edge between  two vertices if and only if the corresponding sets intersect. For any $(j,d)\ne(j',d')$ with $1\le j,j'\le q$ and $d\ne d'$ belonging to $\Dc_1$, by Lemma \ref{lem:K2} we can infer that the following equations both hold:
\begin{align}
\Delta_{j,d}\cap \Delta_{j',d'}=\frac{\Pc_j+d}{N}\cap\frac{\Pc_{j'}+d'}{N}\ne\emptyset
\hspace{0.618cm}& \Leftrightarrow & \frac{\Pc_j\cap K+d}{N}\cap\frac{\Pc_{j'}\cap K+d'}{N}\ne\emptyset
\\
\Delta^\#_{j,d}\cap \Delta^\#_{j',d'}=\frac{\Pc^\#_j+d}{N}\cap\frac{\Pc^\#_{j'}+d'}{N}\ne\emptyset
\hspace{0.618cm}& \Leftrightarrow & \frac{\Pc^\#_j\cap K+d}{N}\cap\frac{\Pc^\#_{j'}\cap K+d'}{N}\ne\emptyset
\end{align}
It then follows that  $\Gc^{(i)}$ and $\Gc^{\#(i)}$ are isomorphic graphs, since $\Pc_k\cap K=\Pc_k^\#\cap K$ for $1\le k\le q$ and hence we have:
\begin{equation}\label{eq:inductive}
\Delta_{j,d}\cap \Delta_{j',d'}\ne\emptyset \quad \Leftrightarrow\quad \Delta^{\#}_{j,d}\cap \Delta^{\#}_{j,d}\ne\emptyset\end{equation}
Since $\Pc^\#_i$ is a component of $K^{(n)}$ and  consists of all $\Delta_{j,d}\in\Uc^{(i)}$, we see that $\Gc^{(i)}$ hence  $\Gc^{\#(i)}$ is a connected graph. This forces the union $\Qc_i$ of all $\Delta^\#_{j,d}\in\Uc^{\#(i)}$ to be a connected set. Since $K^{(n+1)}=\bigcup_i\Qc_i$, it is immediate that $\beta_0\left(K^{(n+1)}\right)=q$. We are done.
\end{proof}

If in Theorem \ref{theo:finite_tau_extended} we have $n_0=2$, then it is easy to check  whether $\beta_0\left(K^{(2)}\right)=\beta_0\left(K^{(3)}\right)$. Notice that the condition  $n_0\ge2$ is sharp. For instance, if $\Dc_1$ is given as in Example \ref{K3_disconnected-1}, then $\beta_0\left(K^{(1)}(3,\Dc_1)\right)=\beta_0\left(K^{(2)}(3,\Dc_1)\right)=1$, while  $\beta_0\left(K^{(3)}(3,\Dc_1)\right)=3$. However, no fractal square $K$ is known such that  $1<\beta_0\left(K^{(1)}\right)=\beta_0\left(K^{(2)}\right)<\beta_0\left(K^{(3)}\right)$.

\begin{exam}\label{K3_disconnected-1}
Let  $\Dc_1=\{(1,0),(0,1),(1,1),(2,1),(2,2)\}$. The left part of Figure \ref{fig:K2K3} illustrates $K^{(n)}(3,\Dc_1)$ for $1\le n\le3$. We have  $\beta_0\left(K^{(1)}(3,\Dc_1)\right)=\beta_0\left(K^{(2)}(3,\Dc_1)\right)=1$ and $\beta_0\left(K^{(3)}(3,\Dc_1)\right)=3$.
\begin{figure}[ht]
%\vspace{-0.5cm}
\begin{center}
\begin{tabular}{cc}
\includegraphics[width=4.0cm]{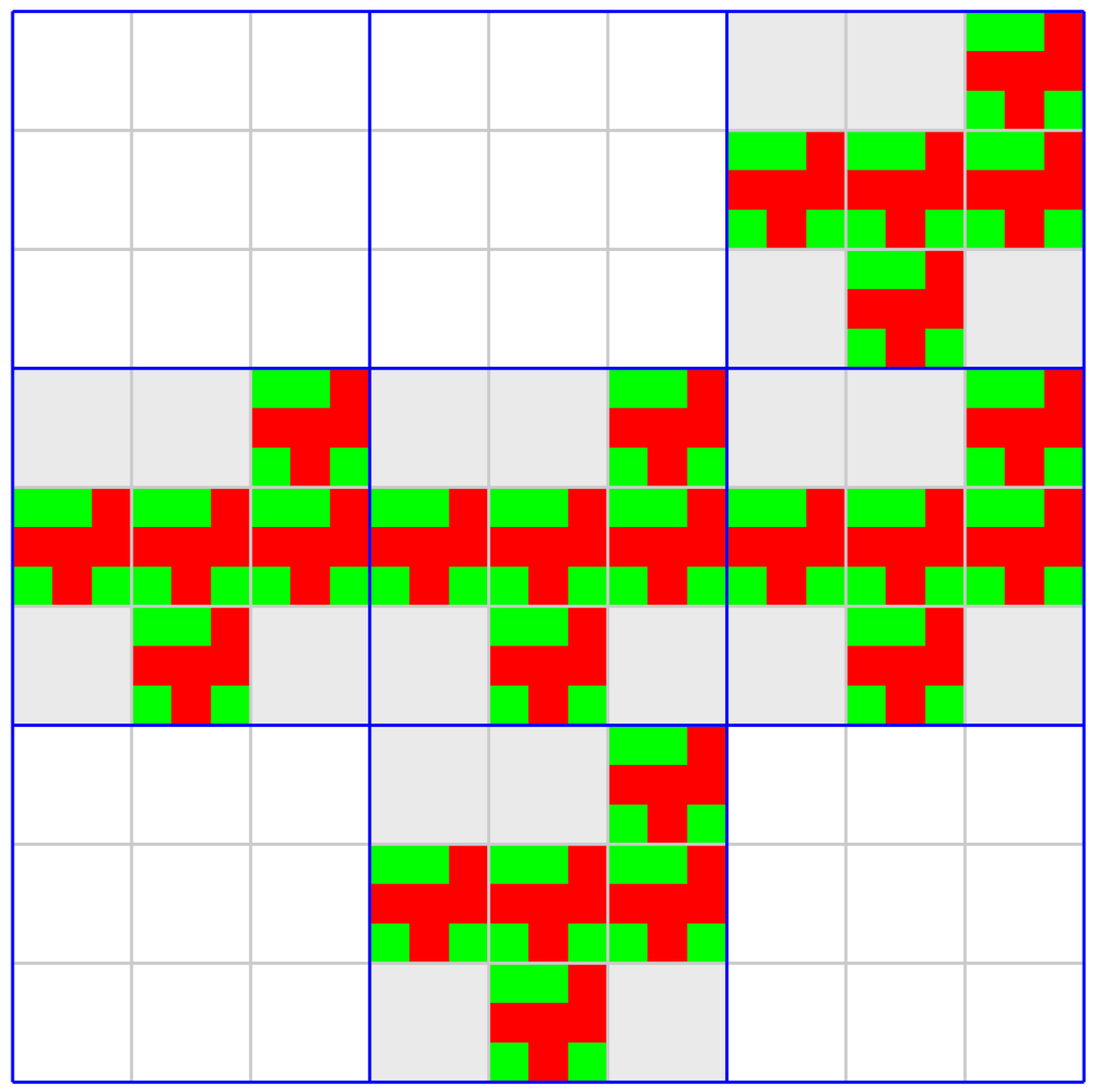} & \hspace{0.618cm}
\includegraphics[width=4.0cm]{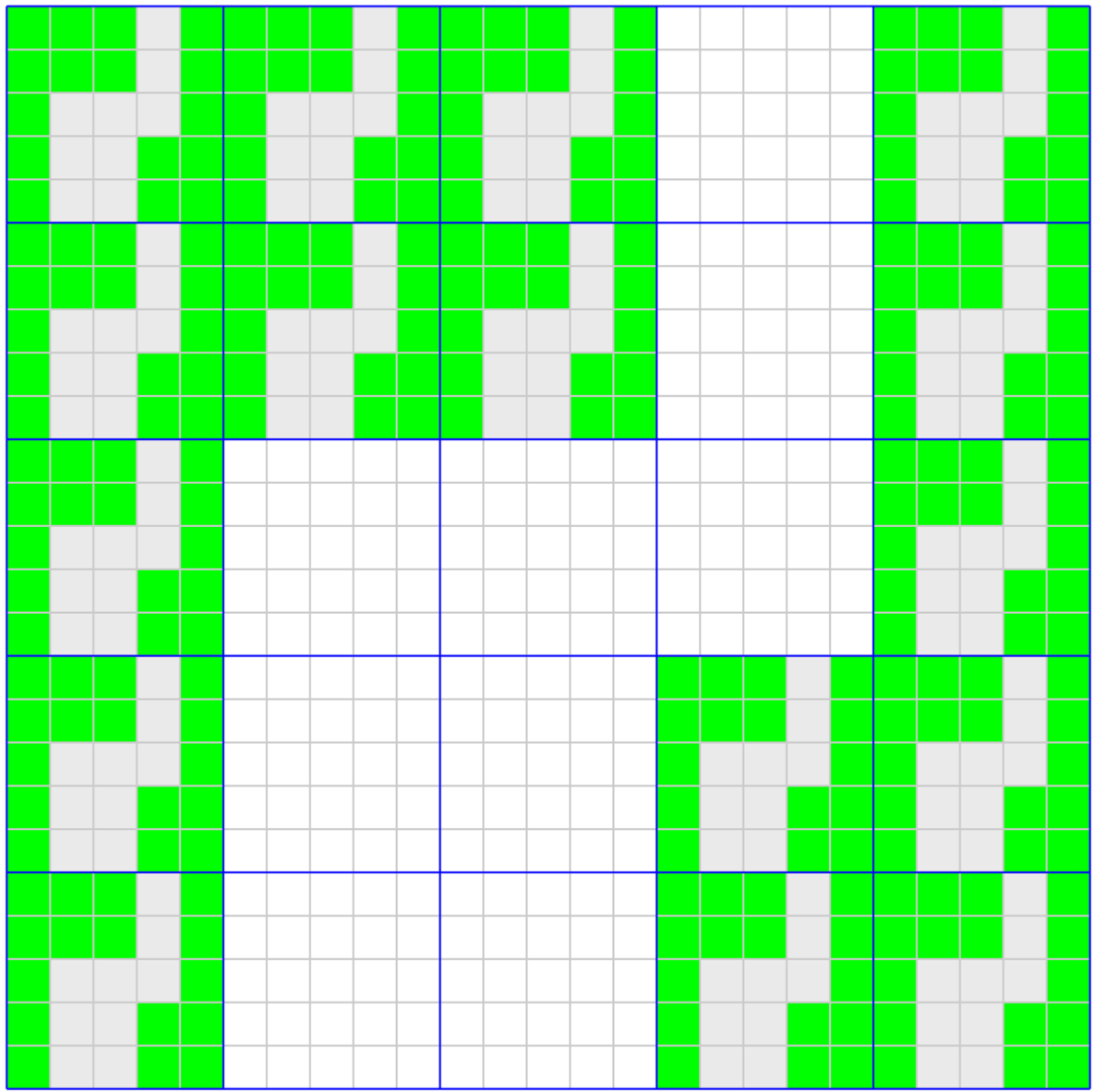}
\end{tabular}
\end{center}
\vspace{-0.618cm}
\caption{The beginning approximations of two fractal squares.}\label{fig:K2K3}
\end{figure}
The right part of Figure \ref{fig:K2K3} illustrates  $K^{(n)}(5,\Dc_2)(n=1,2)$ of $K(5,\Dc_2)$, where $\Dc_2$ is a modification of the one given in \cite[Figure 1]{Xiao21}, with less digits.  Notice that $\beta_0\Big(K^{(1)}(5,\Dc_2)\Big)=\beta_0\Big(K^{(2)}(5,\Dc_2)\Big)=2$. It follows that $\beta_0\Big(K^{(3)}(5,\Dc_2)\Big)=2$ and hence $\beta_0\Big(K(5,\Dc_2)\Big)=2$.  
%First, $d_1\ne d_2\in\Dc$ are incident in the first Hata graph $G_1$ if and only if $d_1+[0,1]^2$ intersects $d_2+[0,1]^2$. Second, the two Hata graphs $G_1$ and $G_2$ both have two components. By \cite[Theorem 1.4]{Xiao21} and \cite[Lemma 1.5]{Xiao21}, $K$ has exactly two components.
\end{exam}

The next theorem gives a sufficient condition for the fundamental group $\pi_1(K)$ to be trivial.
\begin{theo}\label{theo:K3_Pi1}
Given a connected fractal square $K$, $\pi_1(K)$ is trivial whenever  $\pi_1\left(K^{(3)}\right)$ is.
\end{theo}
\begin{proof} Suppose that $K=K(N,\Dc)$, where $\Dc\ne\{0,1,\ldots,N-1\}^2$ is the digit set. If on the contrary $\pi_1(K)$ were not trivial there would be a non-trivial loop $\gamma:[0,1]\rightarrow K$, such that the complement $[0,1]^2\setminus \gamma([0,1])$ has at least one component $U$ that is separated from $\infty$ by $\gamma([0,1])$. Then, for large enough $n$,  at least one square of size $\frac{1}{N^n}$  is entirely contained in  $U\setminus K^{(n)}$. This implies that  $\pi_1\left(K^{(n)}\right)\ne\{0\}$. Therefore, it suffices to show  that $\pi_1\left(K^{(n)}\right)=\{0\}$ for all $n\ge1$. In the sequel, we will do that by induction.

Since $\pi_1\left(K^{(n)}\right)=\{0\}$ with $n\le3$ is ensured by the assumption $\pi_1\left(K^{(3)}\right)=\{0\}$, we suppose that $\pi_1\left(K^{(j)}\right)=\{0\}$ for all $j\le n$ (with $n\ge3$) and use this to infer that  $\pi_1\left(K^{(n+1)}\right)=\{0\}$.

First,  the assumptions $\Dc\ne\{0,1,\ldots,N-1\}^2$ and $\pi_1\Big(K^{(1)}\Big)=\{0\}$  imply the following.

{\bf Claim 1}. $K$ does not contain $\partial[0,1]^2$

Second, we obtain the following.

{\bf Claim 2}.  Given $d\ne d'$ in $\Dc$,  if $\left(K^{(2)}+d\right)\cap\left(K^{(2)}+d'\right)$ is nonempty then it is connected.

Otherwise, there would be a separation, say $\left(K^{(2)}+d\right)\cap\left(K^{(2)}+d'\right)=A\cup B$. Pick $a\in A$ and $b\in B$.  By Lemma \ref{lem:K2}, we have $d-d'\in\{\pm(1,0),\pm(0,1)\}$. We may assume with no loss of generality  that $d-d'=\pm(1,0)$. Pick two arcs $\gamma\subset K^{(2)}+d$ and $\gamma'\subset K^{(2)}+d'$, each of which connects $a$ to $b$. Then $J=\gamma\cup\gamma'$  has a bounded complementary component, say $U$, which is not contained in the union $\left(K^{(2)}+d\right)\cup\left(K^{(2)}+d'\right)$.  This is the case, since  $\pi_1\Big(K^{(2)}\Big)=\{0\}$. It follows that $\frac{U}{N}$ is not contained in $K^{(3)}$.
As $\frac{K^{(2)}+d}{N}$ and $\frac{K^{(2)}+d'}{N}$ are both contained in $K^{(3)}$, $\frac{J}{N}$ is contained in $K^{(3)}$ and $\frac{U}{N}$ is a component of $\bbR^2\setminus\frac{J}{N}$. This is stupid, since $\frac{U}{N}$ is not contained in $K^{(3)}$ and $\pi_1\left(K^{(3)}\right)=\{0\}$ by assumption.

Third, we use {\bf Claim 2} to obtain the following.

%{\bf Claim 2}. There exist no $u,v\in\{0,\ldots, N-2\}$ with   $\left[\frac{u}{N},\frac{u+2}{N}\right]\times \left[\frac{v}{N},\frac{v+2}{N}\right]\subset K^{(1)}$.
% Suppose on the contrary that there were two such integers $u,v$. Then $K^{(3)}$ would contain each of the following sets: $\frac{K^{(2)}+(u,v)}{N},\ \frac{K^{(2)}+(u+1,v)}{N},  \frac{K^{(2)}+(u,v+1)}{N}$ and $\frac{K^{(2)}+(u+1,v+1)}{N}$.  Since $K$ is connected and $\pi_1\left(K^{(3)}\right)$ is trivial by assumption, we may use  {\bf Claim 1} to conclude that $K$ contains the four corners of $[0,1]^2$ and hence the boundary $\partial[0,1]^2$. This forces that $\Dc=\{0,1,\ldots,N-1\}^2$ and contradicts our assumption. % that $\Dc\ne\{0,1,\ldots,N-1\}^2$.

{\bf Claim 3}. If there are $d, d'\in\Dc$ with $d-d'=(1,0)$, then there exists no  $0\le u<v\le N-1$ such that $\{0,N-1\}\times
\{u,v\}\subset\Dc$. Similarly, if there are $d, d'\in\Dc_1$ with $d-d'=(0,1)$, then there exists no  $0\le u<v\le N-1$ such that 
$\{u,v\}\times\{0,N-1\}\subset\Dc$. 
%Consequently, for any $d\ne d'$ in $\Dc$, the intersection  $\left(K^{(n)}+d\right)\cap\left(K^{(n)}+d'\right)$ is either empty or connected.

Suppose on the contrary that there were such integers $u<v$. Then $K^{(2)}$ would contain each of the  sets
$\frac{[0,1]^2\!+(N\!+\!1)d''}{N^2}$ with $d''\in\{0,N-1\}\times\{u,v\}$. See Figure \ref{fig:claim_3}, 
\begin{figure}[ht]
\begin{center}
\includegraphics[width=7.5cm]{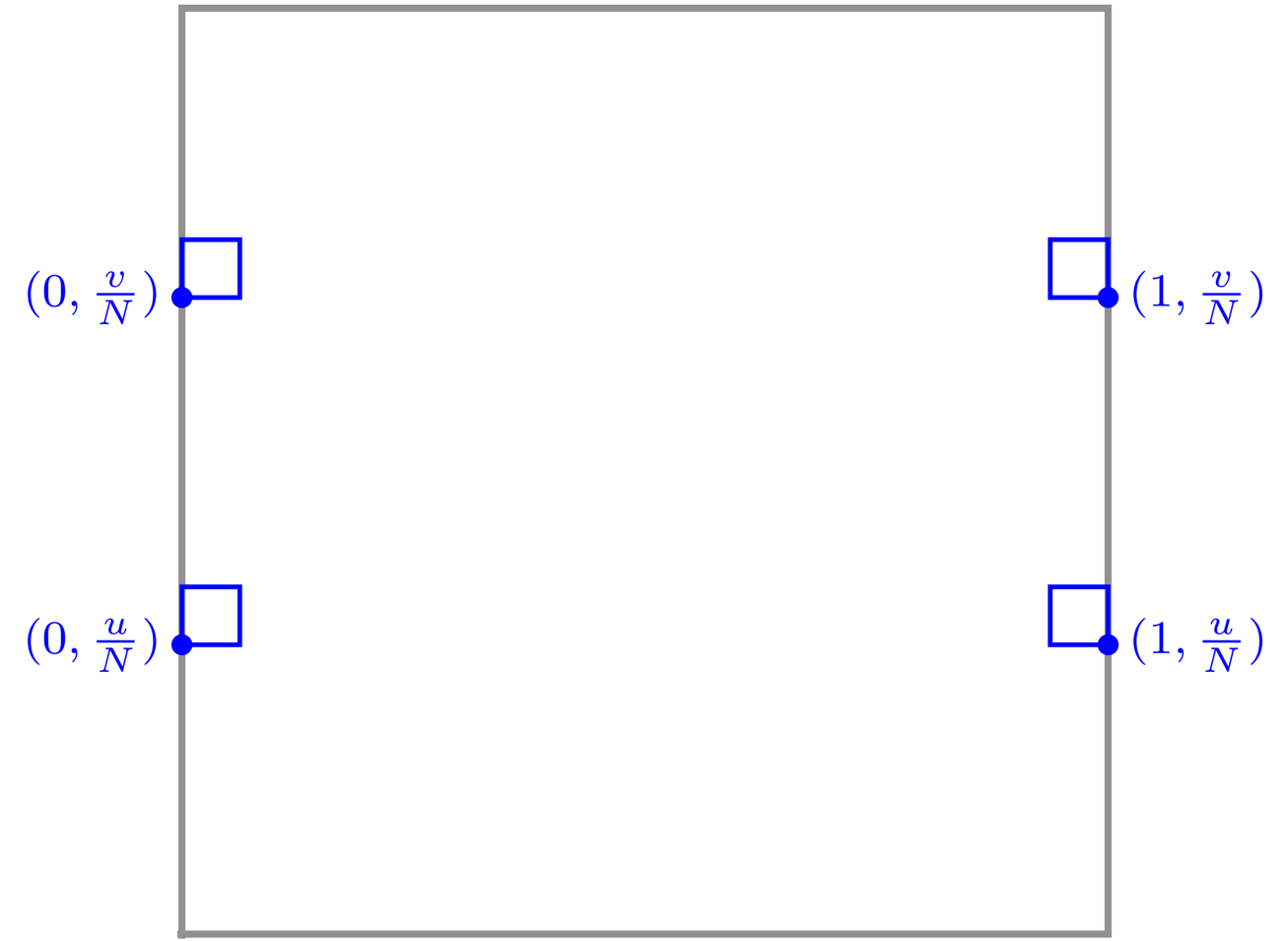}\end{center}\vspace{-0.75cm}
\caption{Locations in $[0,1]^2$ of the four points  
$(0,\frac{u}{N}),(0,\frac{v}{N}),(1,\frac{u}{N}),(1,\frac{v}{N})$.}\label{fig:claim_3}
\end{figure}
which illustrate the relative locations of the four squares $\frac{[0,1]^2+d''}{N}$ with 
$d''\in\{0,N-1\}\times\{u,v\}$. 
Since we assume that there are $d, d'\in\Dc$ with $d-d'=(1,0)$, $\left(K^{(2)}+d\right)\cap\left(K^{(2)}+d'\right)$ contains 
\[
\left\{d+(0,t): \frac{u}{N}+\frac{u}{N^2}\le t\le\frac{u}{N}+\frac{u+1}{N^2}\ \text{or}\ 
\frac{v}{N}+\frac{v}{N^2}\le t\le\frac{v}{N}+\frac{v+1}{N^2}\right\}.
\]
By applying   {\bf Claim 2} we may conclude  that $K$ contains the four corners of $[0,1]^2$ and hence $\partial[0,1]^2$. This contradicts {\bf Claim 1}.

Now we complete the proof by pointing out three basic observations. First,  $\left\{\frac{K^{(n)}+d}{N}: d\in\Dc\right\}$ is a cover of $K^{(n+1)}$. Second,  from {\bf Claim 2} and {\bf Claim 3} it follows that for any $d\ne d'$ in $\Dc$ and any  $n\ge2$  the intgersection $\frac{K^{(n)}+d}{N}\bigcap\frac{K^{(n)}+d'}{N}$  is either empty or a singleton. Third, due to the van Kampen Theorem (Ref. \cite[p.43, Theorem 1.20]{Hatcher02}), it is immediate that $\pi_1\left(K^{(n+1)}\right)=\{0\}$. \end{proof}

To conclude this section, we give three fractal squares. The one in Example \ref{exmp:K_and_Pi} shows that the converse of Theorem \ref{theo:K3_Pi1} is not true. The one given in Example \ref{exmp:no_line_dendrite} is a dendrite, since it satisfies the conditions of Theorem \ref{theo:K3_Pi1}. However, the one given in Example \ref{exmp:no_line} has a nontrivial fundamental group hence  does not satisfy the conditions of Theorem \ref{theo:K3_Pi1}. Here we note that neither of the two fractal squares given in Examples \ref{exmp:no_line_dendrite} and \ref{exmp:no_line} contains a line segment.

\begin{exam}\label{exmp:K_and_Pi}
Let $K=K_{\Dc}$ be a fractal square of order five, whose first and sixth approximations are illustrated in the left and the middle of Figure \ref{fig:K3_Pi1}. Notice  that none of $\pi_1\left(K^{(n)}\right)$ is trivial while $\pi_1(K)$ is. Let $d'\notin \Dc$ be the digit for which  $f_{d'}[0,1]^2=\frac{[0,1]^2+d'}{5}$ is depicted in the left of Figure \ref{fig:K3_Pi1}. Then the resulting fractal square $K'$ has a nontrivial hence uncountable fundamental group. See the right of Figure \ref{fig:K3_Pi1} for the sixth approximation of $K'$.
\begin{figure}[ht]
\begin{center}\begin{tabular}{ccc}
\includegraphics[width=5.5cm]{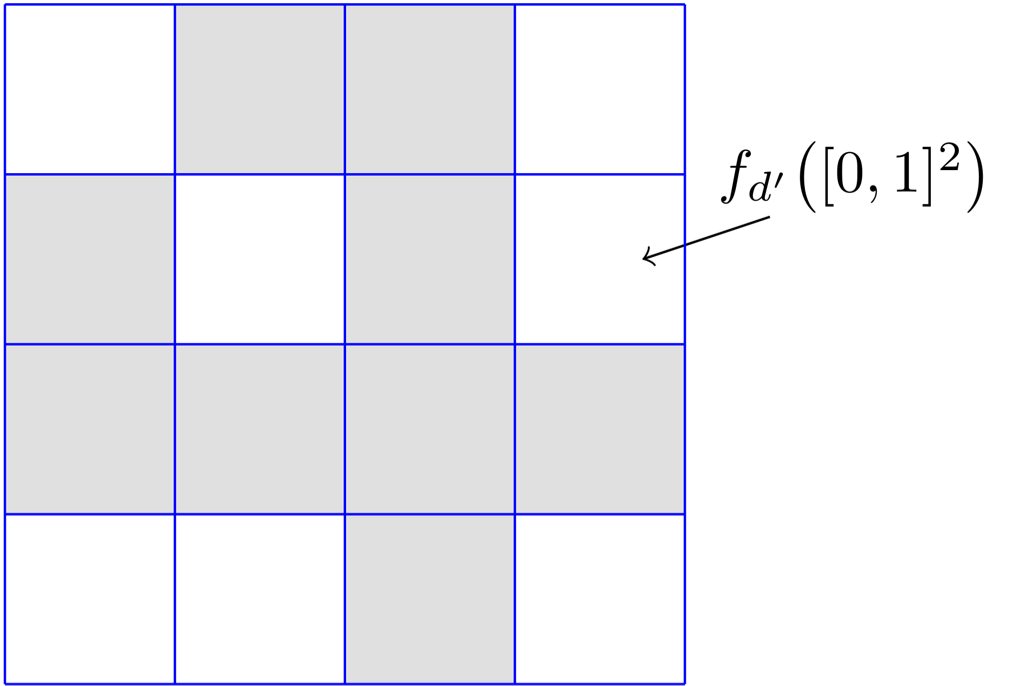}& \includegraphics[width=3.75cm]{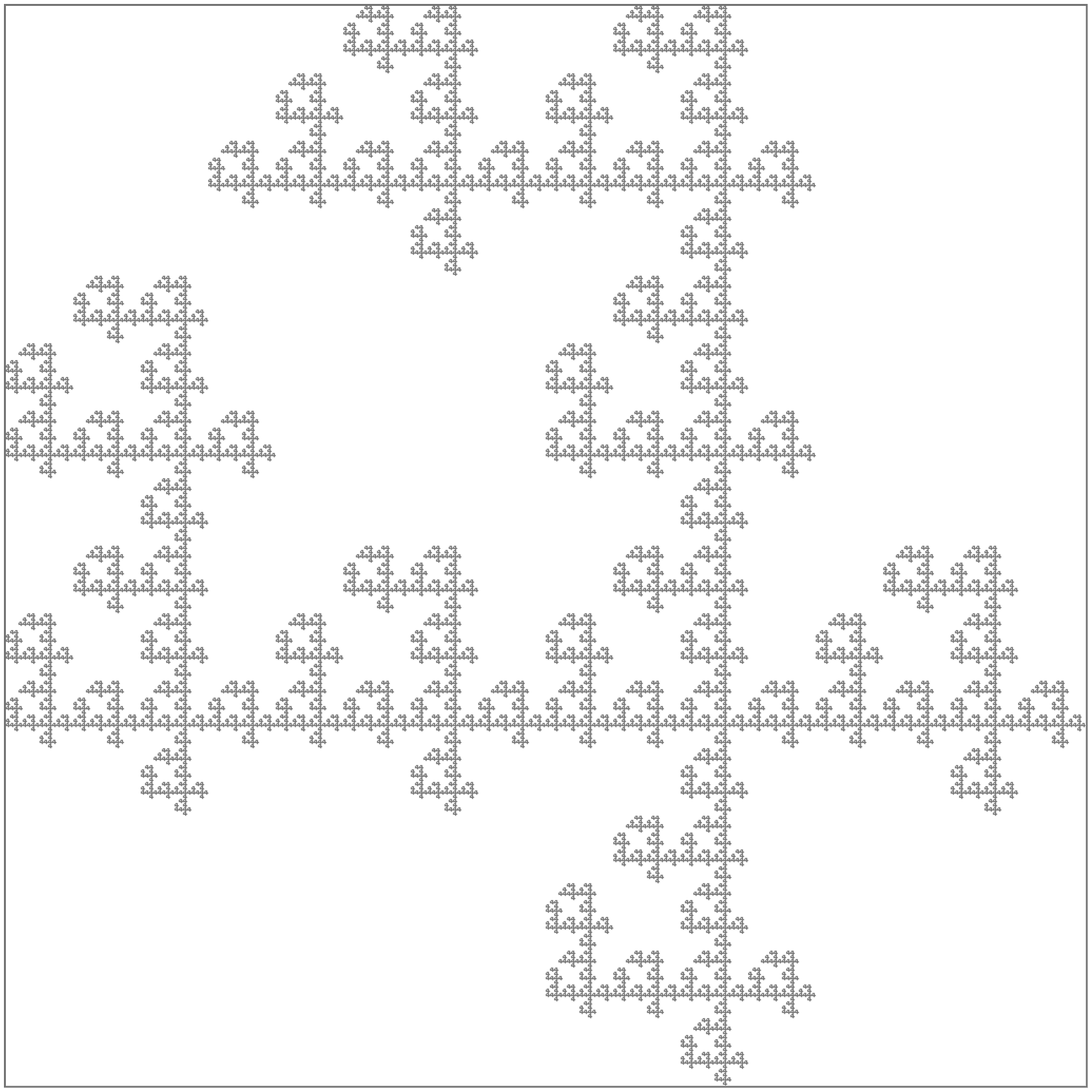}
& \hspace{0.5cm} \includegraphics[width=3.75cm]{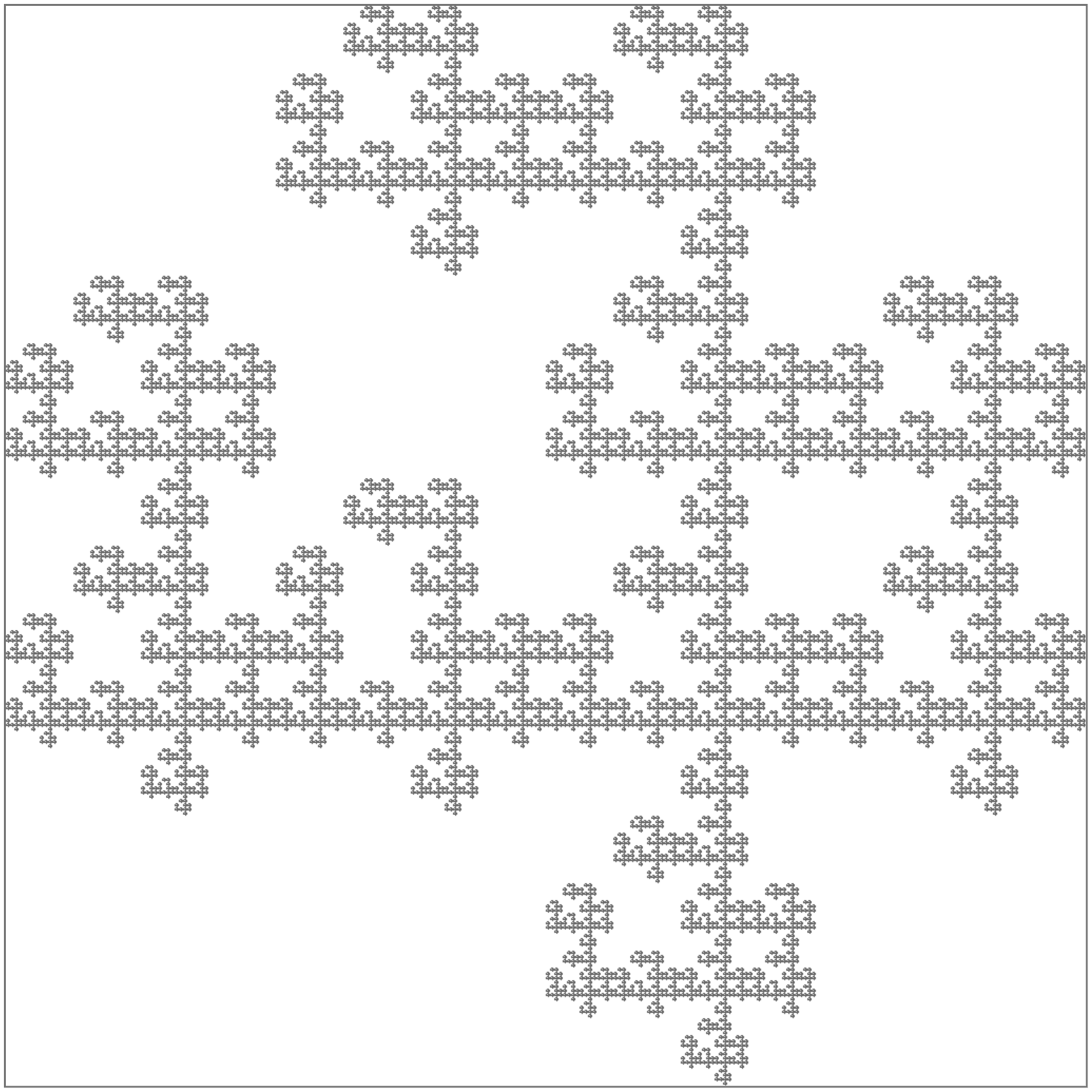}
\end{tabular}
\end{center}\vspace{-0.75cm}
\caption{ $K^{(1)}$ and $K^{(6)}$ (left and middle); the sixth approximation for $K'$.}\label{fig:K3_Pi1}
\end{figure}
\end{exam}

\begin{exam}\label{exmp:no_line_dendrite}
Let $K=K_{\Dc}$ be a fractal square of order four, whose first approximation and first Hata graph $G_1$ are illustrated in the left of Figure \ref{fig:noline_dendrite}. It is easy to check that (1) $G_1$ is a tree and (2) $\pi_1\left(K^{(1)}\right)=\pi_1\left(K^{(2)}\right)=\{0\}$. See the middle of Figure \ref{fig:noline_dendrite} for the approximations 
$K^{(1)}\supset K^{(2)}$ and the intersection of $\frac{K^{(1)}+d}{4}$ with $\frac{K^{(1)}+d'}{4}$ for $d\ne d'$ lying in $\Dc$. From these details, it is easy to further infer that $\pi_1\left(K^{(3)}\right)=\{0\}$. By Theorem \ref{theo:K3_Pi1},  $\pi_1(K)$ is trivial hence $K$ is a dendrite. Notice that there is no line segment contained in $K$.
\begin{figure}[ht]
\begin{center}
\begin{tabular}{cc}
\includegraphics[width=8cm]{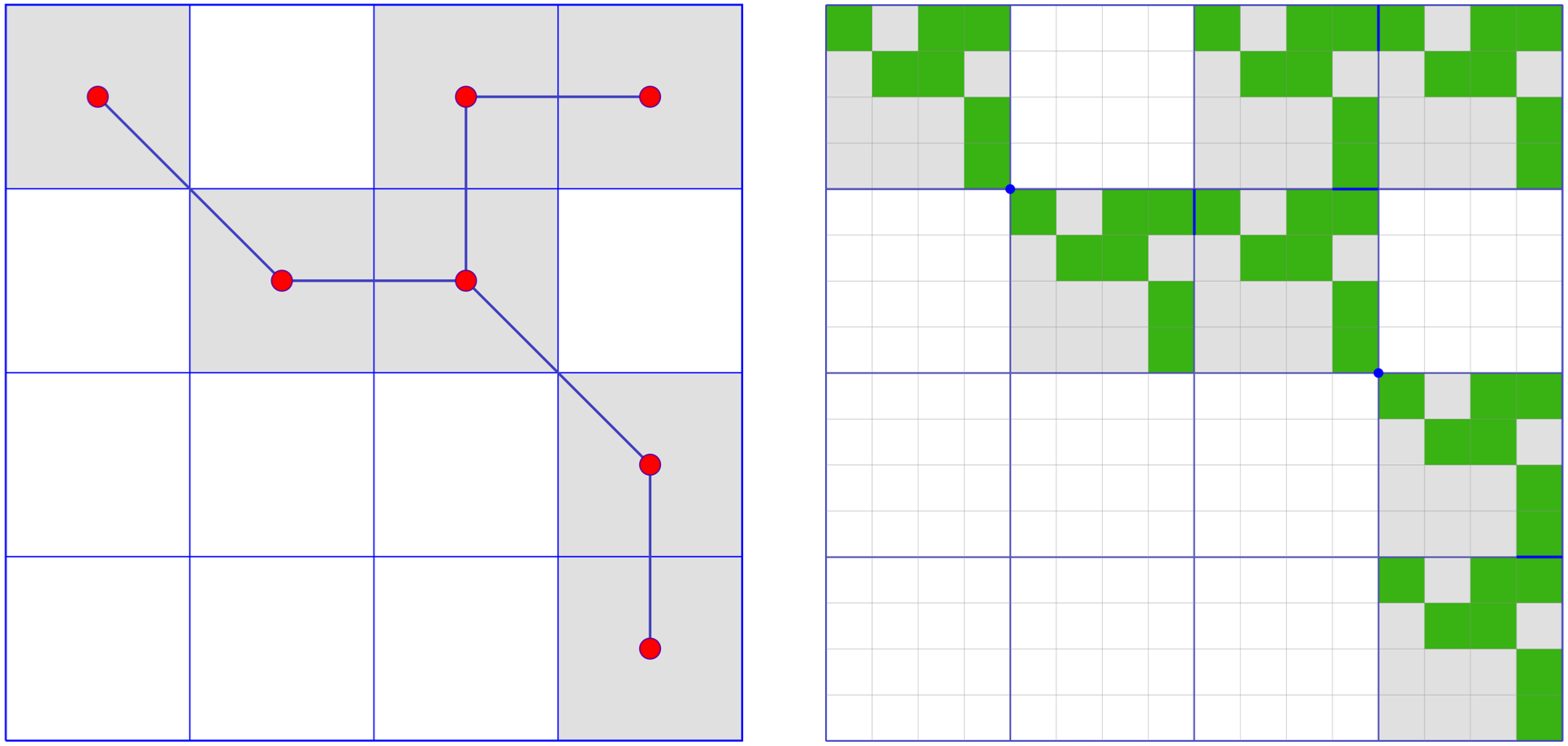}
& \includegraphics[width=3.8cm]{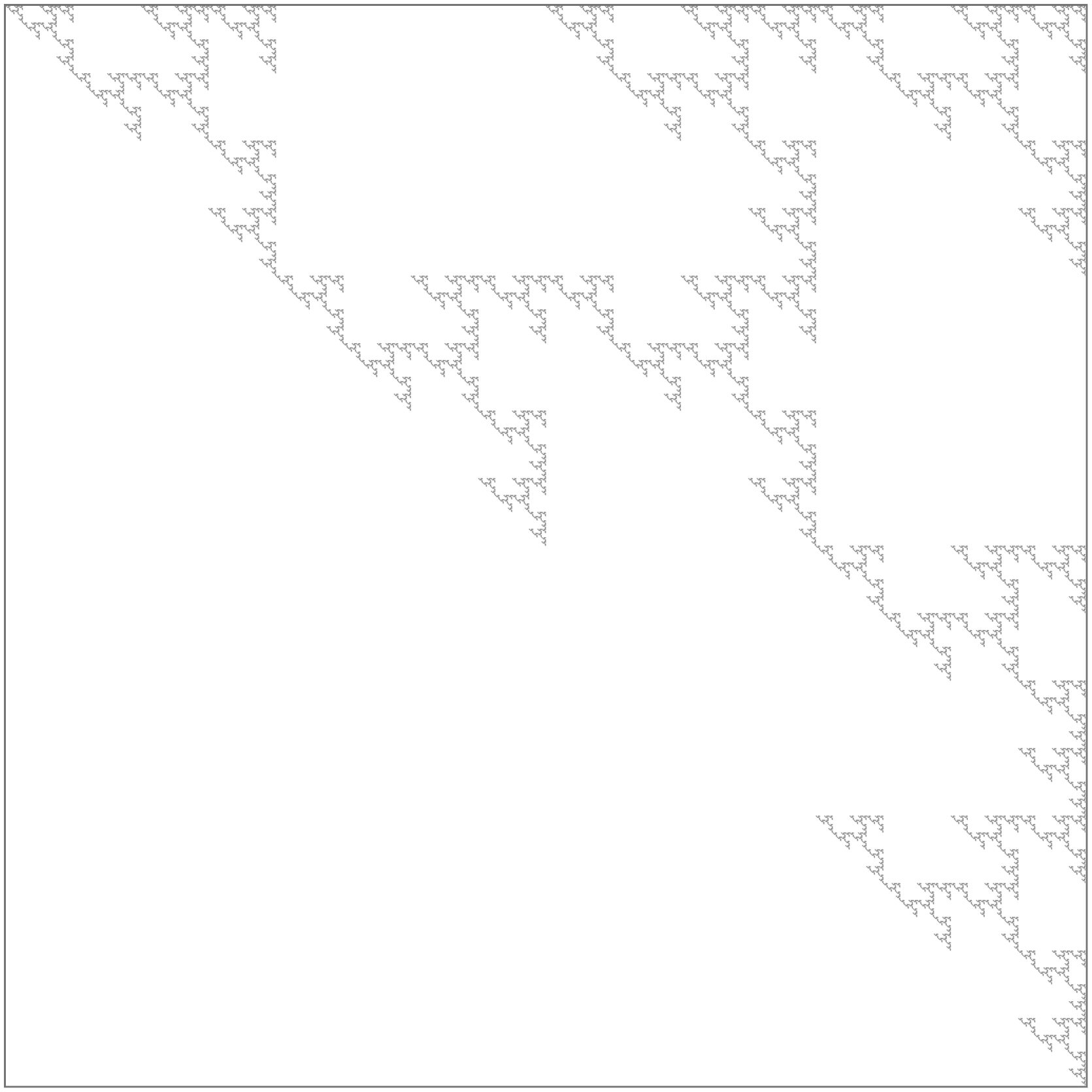}
\end{tabular}
\end{center}\vspace{-0.75cm}
\caption{$K^{(1)}$ and $G_1$ {\bf(left)}; \ $K^{(1)}\supset K^{(2)}$ {\bf(middle)}; and $K^{(6)}$ {\bf(right)}.}\label{fig:noline_dendrite}
\end{figure}
\end{exam}

\begin{exam}\label{exmp:no_line}
Let $K=K_{\Dc}$ be a fractal square of order four, whose first approximation and first Hata graph $G_1$ are illustrated in the left of Figure \ref{fig:noline_nondendrite}. It is easy to check that (1) $G_1$ is a tree and (2)  $\pi_1\left(K^{(1)}\right)=\{0\}$ while $\pi_1\left(K^{(n)}\right)\ne\{0\}$ for all $n\ge2$. See the middle of Figure \ref{fig:noline_nondendrite} for the approximations 
$K^{(1)}\supset K^{(2)}$. Notice that for certain $d\ne d'$ lying in $\Dc$,  the intersection of $\frac{K^{(1)}+d}{4}$ with $\frac{K^{(1)}+d'}{4}$ may be disconnected. Moreover, we can infer that $\pi_1(K)\ne\{0\}$. Again, there is no line segment contained in $K$.
\begin{figure}[ht]
\begin{center}
\begin{tabular}{cc}
\includegraphics[width=8cm]{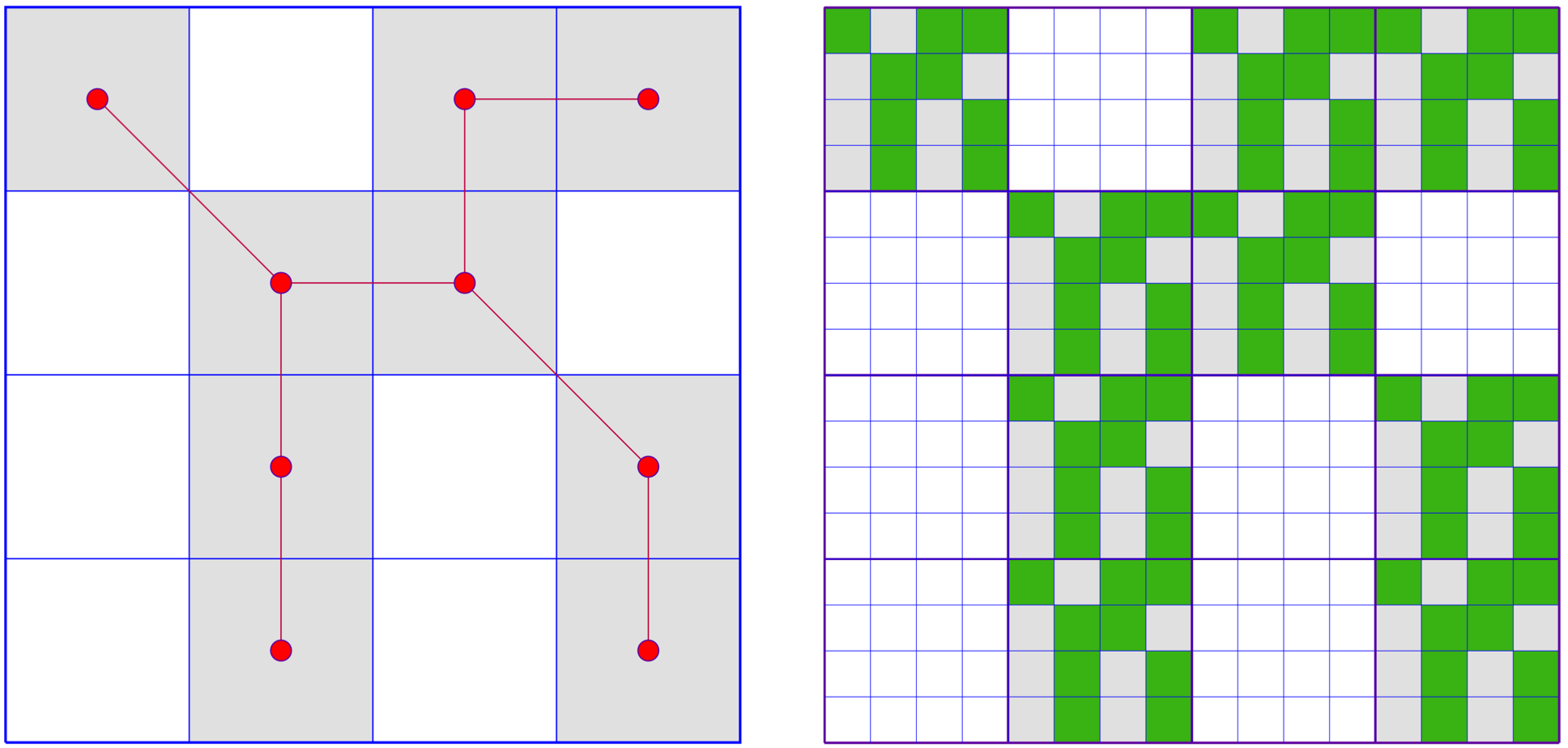}
& \includegraphics[width=3.8cm]{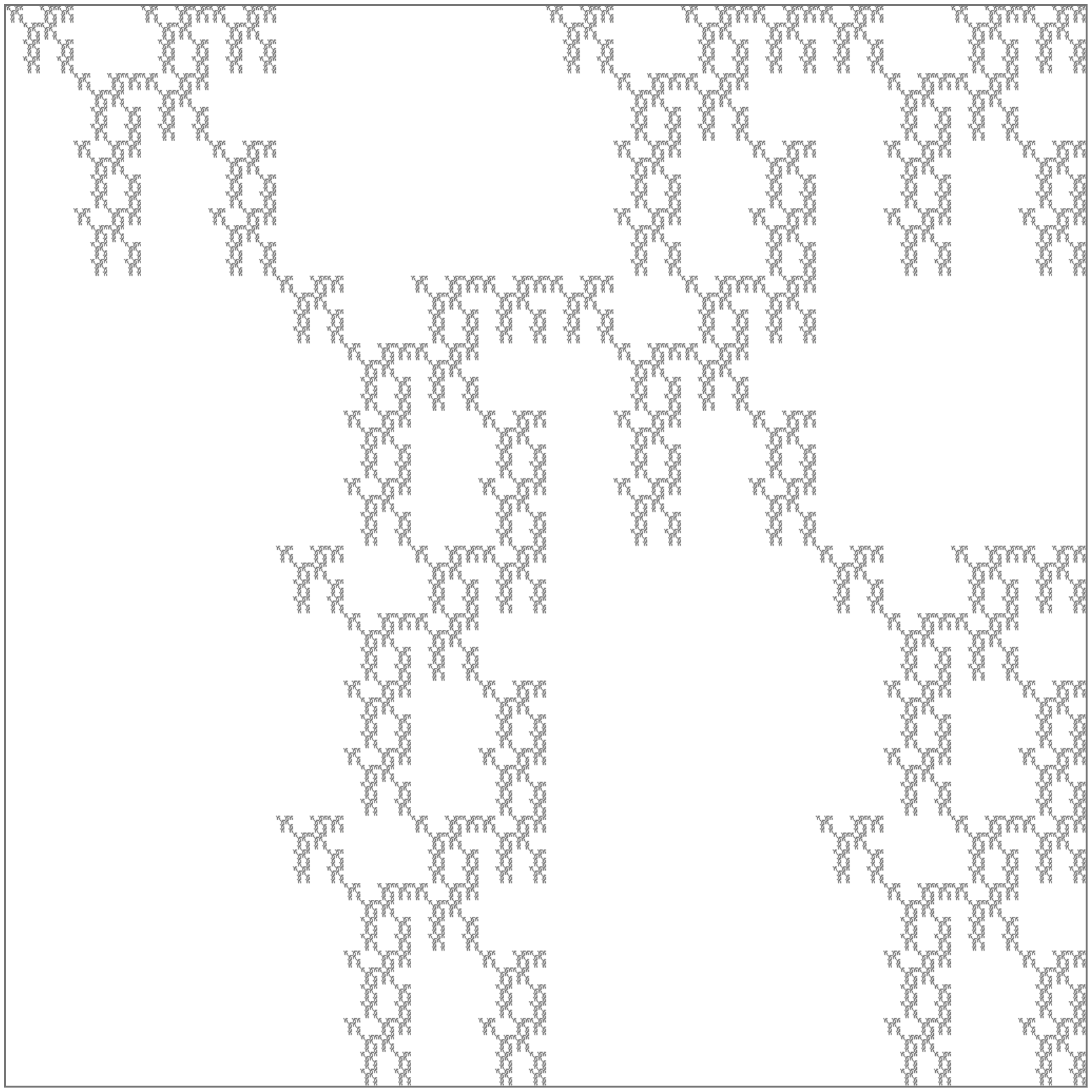}
\end{tabular}
\end{center}
\vspace{-0.75cm}
\caption{$K^{(1)}$ and $G_1$ {\bf(left)}; \ $K^{(1)}\supset K^{(2)}$ {\bf(middle)}; and $K^{(6)}$ {\bf(right)}}\label{fig:noline_nondendrite}
\end{figure}
\end{exam}

\section{Proving Theorem \ref{maintheo:LLR_improved}}\label{proof_1}

This section aims to prove Theorems \ref{maintheo:LLR_improved} and \ref{maintheo:product_form}. Let us first recall   \cite[Theorem 2.2]{Lau-Luo-Rao13}, which states that for any fractal square $K$ and the orbit $H=K+\bbZ^2$ there are exactly two possibilities: either all components of $\bbR^2\setminus H$ are unbounded or each of them is of diameter $\le\sqrt{2}(N^2+1)^2/N$.
In the first, all components of $K$ are points or line segments. So all atoms of $K$ are points or line segments hence $\lambda_K(K)\subset\{0,1\}$. In the second, $H$ has at least one component that separates the plane.
In order to obtain Theorem \ref{maintheo:LLR_improved}, we just need to show that $K$ is a Peano compactum in the latter case.
To do that, we will need the following.
\begin{lemm}[{\bf \cite[Theorem 3]{LLY-2019}}]\label{lem:LLY19-Thm3}
A compact set $K\subset\bbR^2$ is a Peano compactum if and only if it fulfills the Sch\"onflies condition, so that for any strip $U$  bounded by two parallel infinite lines $L_1$ and $L_2$ there are at most finitely many components in $\overline{U}\setminus K$ that intersect $L_1$ and $L_2$ both.
\end{lemm}

Hereafter in this section, $K=K(N,\Dc)$ ia  a fractal square of order $N\ge3$ such that that all components of $\bbR^2\setminus H$ have a diameter $\le\sqrt{2}(N^2+1)^2/N$. Moreover, $H=K+\bbZ^2$.

\begin{proof} [{\bf Proof for Theorem \ref{maintheo:LLR_improved}}]
If $K$ were not a Peano compactum then by Lemma \ref{lem:LLY19-Thm3} we could find two parallel lines $L_1$ and $L_2$ that bound an unbounded strip $U$, with $\partial U=L_1\cup L_2$, such that $\overline{U}\setminus K$ has infinitely many components, say $A_i$, that intersect both $L_1$ and $L_2$. In each $A_i$ we may pick a Jordan arc $\alpha_i$ joining a point $a_i\in L_1$ to a point $b_i\in L_2$. Moreover, we may require that $\alpha_i\cap(L_1\cup L_2)=\{a_i,b_i\}$. Going to an appropriate subsequence, if necessary, we may assume that $\lim\limits_{i\rightarrow\infty}\alpha_i=\alpha_\infty$ under the Hausdorff distance and $a_\infty=\lim\limits_{i\rightarrow\infty}a_i$, $b_\infty=\lim\limits_{i\rightarrow\infty}b_i$. For $r>0$ small enough, we may pick some integer $n_0>0$ and some $i\in\left\{0,1,\ldots, N^{n_0}-1\right\}$ and  separate the closed disks $D_r(a_\infty)$ and $D_r(b_\infty)$ by a strip $W$, which is either of the form 
$\bbR\times\left[\frac{i}{N^{n_0}},\frac{i+1}{N^{n_0}}\right]$ or of the form $\left[\frac{i}{N^{n_0}},\frac{i+1}{N^{n_0}}\right]\times\bbR$.  
We may assume that $W=\bbR\times\left[\frac{i}{N^{n_0}},\frac{i+1}{N^{n_0}}\right]$ and consider $\alpha_i$ as a map of $[0,1]$ into $\overline{U}$, with $\alpha_i(0)=a_i, \alpha_i(1)=b_i$. So we can pick $t_i<s_i$ such that $a_i^*=\alpha_i(t_i)\in L'$, $b_i^*=\alpha_i(s_i)\in L''$, and $\alpha_i([t_i,s_i])\subset W$. Here $\displaystyle L'=\bbR\times\left\{\frac{i}{N^{n_0}}\right\}$ and $\displaystyle L''=\bbR\times\left\{\frac{i+1}{N^{n_0}}\right\}$. See Figure \ref{fig:Find_W} for relative locations of $L',L'',L_1,L_2$.
\begin{figure}[ht]
\vspace{-0.2cm}
\center{
\includegraphics[width=8.5cm]{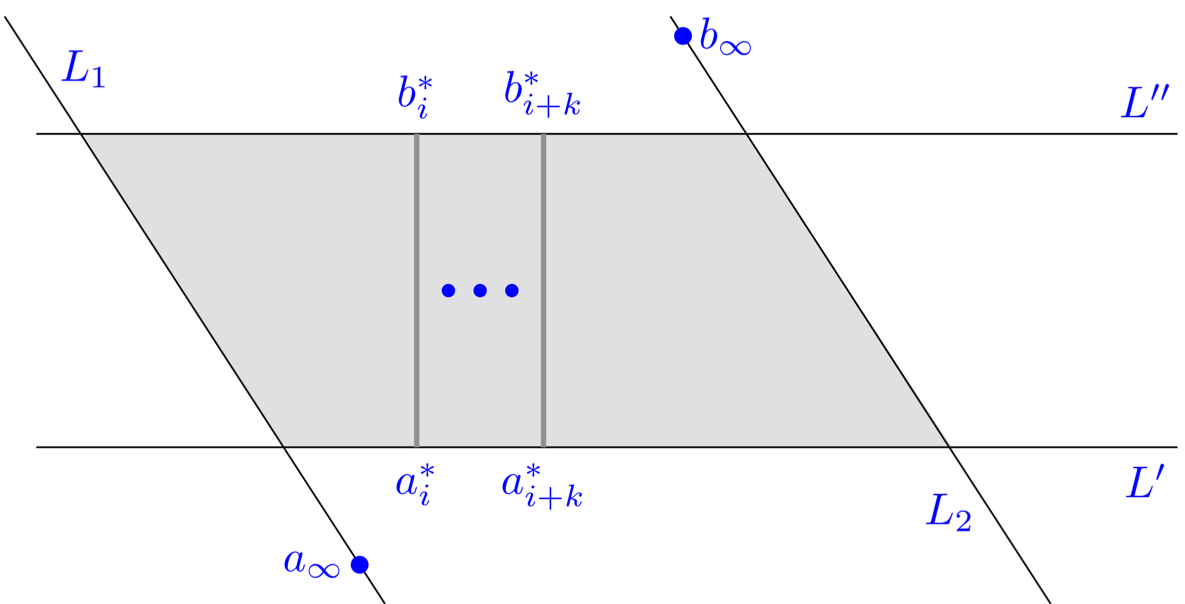}
}
\vspace{-0.5cm}
\caption{Relative locations of the points $a_\infty,b_\infty, a_i^*, b_i^*$ and the lines $L_1,L_2, L',L''$.}\label{fig:Find_W}
\end{figure}
Let $\alpha_i^*\subset\alpha_i$ be the sub-arc from $a_i^*$ to $b_i^*$.
Then, no two of the arcs $\alpha_i$ can be contained in a single component of $W\setminus K$. It follows that $W\cap K$ has infinitely many components, say $Q_k$, that intersect $L'$
and $L''$ both. See for instance \cite[Lemma 3.8]{LLY-2019}.

Now, let $V$ be the only component of $W\setminus(L'\cup L'')$ whose boundary intersects $L'$ and $L''$ both. In Figure \ref{fig:Find_W} the region  $V$ is represented as a shadowed parallelogram and $\alpha_i^*$ as a vertical segment.

It is immediate that  all but finitely many of the arcs $\alpha_i^*$ is contained in $\overline{V}$. By going to an appropriate subsequence, if necessary, we may assume that $\lim\limits_{k\rightarrow\infty}Q_k=Q_\infty$ under the Hausdorff distance. Here, the limit continuum $Q_\infty$ is necessarily contained in a single component  of $W\cap K$, to be denoted by $Q_0$.
See Figure \ref{fig:strip_W} for a natural order among the connected sets $Q_k$.
\begin{figure}[ht]
\vspace{-0.2cm}
\center{
\includegraphics[width=9cm]{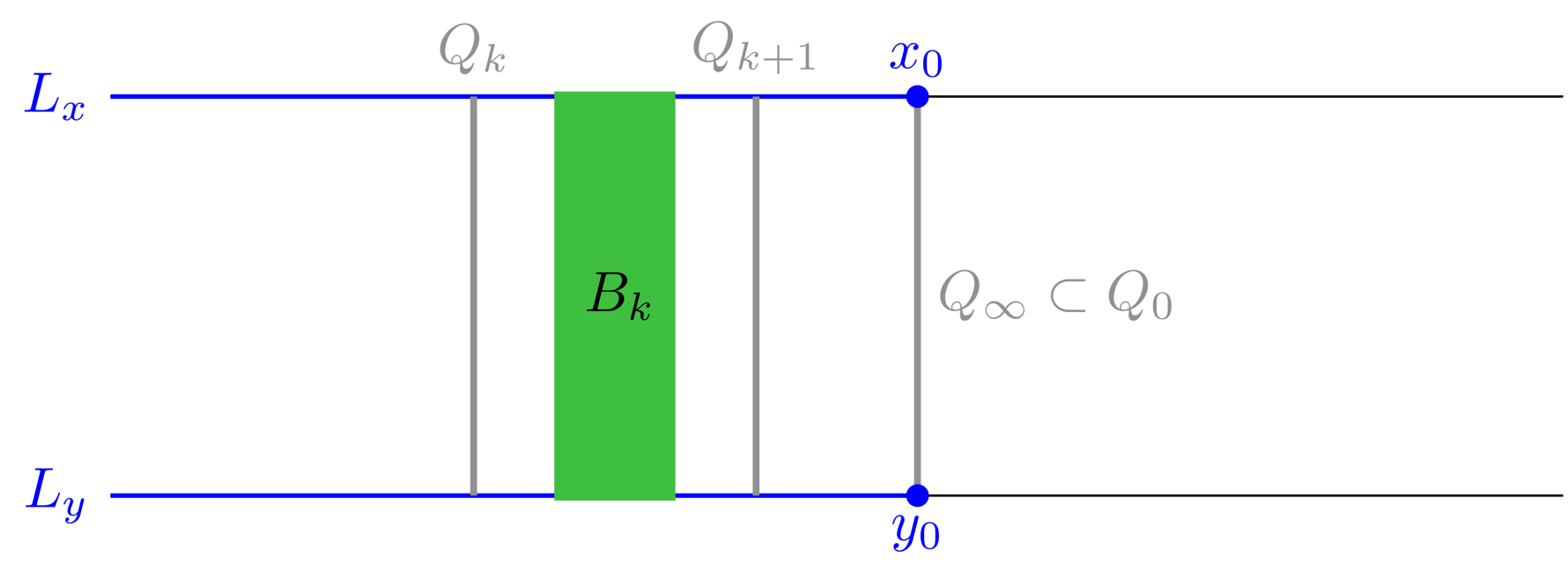}
}
\vspace{-0.5cm}
\caption{The relative locations of $Q_k, B_k, Q_{k+1}$ and $Q_\infty$.}\label{fig:strip_W}
\end{figure}
Notice that $W\setminus Q_0$ has two unbounded components, one is to the left of $Q_\infty$ and the other to the right. By going to an appropriate subsequence, if necessary, we may assume that every $Q_k$ is to the left of $Q_{k+1}$. See Figure \ref{fig:strip_W}. Therefore, we can find infinitely many components $B_k(k\ge1)$ of $W\setminus K$  intersecting $L'$
and $L''$ both.

Here we choose $B_1$ so that it separates $Q_1$ and $Q_0$ in $W$. Moreover, given $B_1,\ldots,B_k$ for any $k\ge1$  we may further choose $B_{k+1}$ so that it separates $\left(\bigcup_{j=1}^{k+1}Q_j\right)\cup\left(\bigcup_{j=1}^{k}B_j\right)$ and $Q_0$ in $W$.
Clearly, $\lim\limits_{k\rightarrow\infty}B_k=\lim\limits_{k\rightarrow\infty}Q_k=Q_\infty$ under the Hausdorff distance.

Given $n>n_0$,  we consider all the $n$-cells (see Definition \ref{def:approx}) that intersect infinitely many $Q_k$. Such a cell is called a {\bf black cell of order $n$}. The other $n$-cells are referred to as the {\bf white cell of order $n$}. Notice that every black cell intersects infinitely many $B_k$.

Pick $x_0\in (Q_\infty\cap L'')$
and $y_0\in  (Q_\infty\cap L')$, such that the left component of
$L''\setminus\{x_0\}$ (denoted by $L_x$) and that of
$L'\setminus\{y_0\}$ (denoted by $L_y$) are disjoint from $Q_0$. Then, both of them intersect all $Q_k$ and all $B_k$. See Figure \ref{fig:strip_W}.

Let $M_n\subset K^{(n)}$ consist of all the black cells of order $n$. Notice that the interior $M_n^o$ has at most finitely many components (say $W_1,\ldots,W_p$), such that $\overline{W_i}\cap\overline{W_j}$ for any $i\ne j$ is a finite set. Using this fact,  one may infer that $M_n^o$ is actually connected. See  \cite[Lemma 4.5]{LRX-2022}. Therefore, $M_n$ is a continuum that contains $Q_\infty$ and hence all but finitely many $B_k$. Notice that $M_n$ has no cut point. Moreover, by Torhorst Theorem (\cite[p.106, (2.2)]{Whyburn42}), one of the two unbounded components of $W\setminus M_n$ is bounded by $L_x\cup L_y\cup\beta_n$, where $\beta_n\subset \partial M_n$ is an open arc lying in $W$  and joining a point on $L_x$ to a point on $L_y$.

Let $U_n$ be the bounded component of $\bbR^2\setminus\left(\beta_n\cup L_x\cup L_y\cup Q_0\right)$ with $\beta_n\subset\partial U_n$. Then it is routine to infer that all but finitely many $B_k$ are contained in the closure of $U_n$. Since $\bbR^2\setminus M_n$ has at most finitely many components, we can further infer that all but finitely many $B_k$ are contained in  $M_n^{o}\cup \partial W$.

Fix some $B_k$ with this property and pick an arc $\gamma_n\subset B_k$ that intersects  $L_x$ and $L_y$ both. Since $B_k\subset K^{(n)}\setminus K$, we see  that $\gamma_n$ is contained in the interior of $K^{(n)}\setminus K$ and that $N^n\gamma_n\subset(\bbR^2\setminus H)$. See \cite[Lemma 4.3]{LRX-2022}. Since $n>n_0$ is flexible and since the diameter of $\gamma_n$ is greater than $N^{-n_0}$, we may choose large enough $n$ such that the diameter of $N^n\gamma_n$ is greater than  $\sqrt{2}(N^2+1)^2/N$. This is absurd, since all components of $\bbR^2\setminus H$ are of diameter $\le\sqrt{2}(N^2+1)^2/N$.
\end{proof}

%Theorem \ref{maintheo:LLR_improved} improves both \cite[Theorem 2.2]{Lau-Luo-Rao13} and \cite[Theorem 2]{LRX-2022}.
Notice that the above proof  uses \cite[Theorem 2.2]{Lau-Luo-Rao13} and  is independent of \cite[Theorem 2]{LRX-2022}. In the rest of this section, we deal with Theorem \ref{maintheo:product_form}. Let us start from some terminology.

\begin{deff}\label{def:product_form}
A fractal square $K=K(N,\Dc)$ of order $n$ (or its digit set $\Dc$) is said to be {\bf of product form} provided that there exist $i_1<\cdots<i_j$ in $\{0,1,\ldots,N-1\}$ such that
%\[\displaystyle K=\bigcup\left\{\frac{K+d}{N}: \ d\in\Dc\right\} \]
$\Dc=\{i_1,\ldots,i_j\}\times\{0,1,\ldots,N-1\}$ or $\Dc=\{0,1,\ldots,N-1\}\times\{i_1,\ldots,i_j\}$. \end{deff}

%See Figure \ref{cantor*line} for the two choices of $\Dc$ such that the resulting fractal square is congruent to the product of Cantor ternary set and $[0,1]$. \begin{figure}[ht] \begin{center} \includegraphics[width=9cm]{cantor_line.png} \end{center}\vspace{-0.5cm} \caption{The digit sets for two fractal squares of the product form.}\label{cantor*line} \end{figure}
A fractal square $K$ with product form is the product of $[0,1]$ and a homogeneous linear Cantor set $E$ with ratio $\frac1N$. Such a set is just the attractor of an IFS consisting of $2\le q\le n-1$ maps of the form $g_i(x)=\frac{t+s_i}{N}$ with $s_i\in\{0,1,\ldots,N-1\}$.

Theorem \ref{maintheo:product_form} is implied by the following.
\begin{theo}\label{theo:product}
 A fractal square $K=K(N,\Dc)$ has a product form if and only if $\lambda_K(K)=\{1\}$
\end{theo}
\begin{proof}
The ``IF'' part is trivial. So we just prove the ``ONLY IF'' part. We will show that if $K$ is not of product form then $\lambda_K(K)\ne\{1\}$.

If $\Dc$ contains $\{i\}\times\{0,1,\ldots,N-1\}$ or $\{0,1,\ldots,N-1\}\times\{i\}$ for some $i$, then $0\in\lambda_K(K)$. For instance, if $\Dc\supset\{i\}\times\{0,1,\ldots,N-1\}$, then there exists $j\ne i$ with $\Dc\nsubseteq\{j\}\times\{0,1,\ldots,N-1\}$. From this we can infer that for any $d\in\Dc\cap\{j\}\times\bbR$ the fixed point of $f_d(x)=\frac{x+d}{N}$, denoted as $x_0$, is a point component of $K$.  This is due to two basic observations. First, the common part of $K$ and the line $\{x_0\}\times\bbR$ is a totally disconnected set, denoted by $X_0$. Second, every point of $X_0$ is a point component of $K$, since every non-degenerate component of $K$ is a vertical line segment. Consequently, we have $\lambda_K(x_0)=0\in\lambda_K(K)$.

The same argument still works, if $\Dc$ contains $\{0,1,\ldots,N-1\}\times\{i\}$.

In the following, we consider that case that for any $i$ neither of $\{i\}\times\{0,1,\ldots,N-1\}$ and $\{0,1,\ldots,N-1\}\times\{i\}$ is contained in $\Dc$. In such a case, $H$ does not contain any horizontal or vertical infinite lines.
We need to prove that  $\lambda_K(K)\ne\{1\}$.
If $\lambda_K^{-1}(1)=\emptyset$, we are done. So we only consider the case  $\lambda_K^{-1}(1)\ne\emptyset$. Then, it will suffice to show that $0\in\lambda_K(K)$.

By Theorem \ref{maintheo:LLR_improved},  all components of $\bbR^2\setminus H$ are unbounded and every component of $K$ is a singleton or a line segment. Moreover, $H$ contains straight lines of the same slope $\tau\in\bbQ\setminus\{0\}$. In particular, we can fix such a straight line  $l$ that cuts the interior of $[0,1]^2$ into two regions.
See Figure \ref{corner}.
\begin{figure}[ht]
\vspace{-0.25cm}
\center{
\includegraphics[width=6cm]{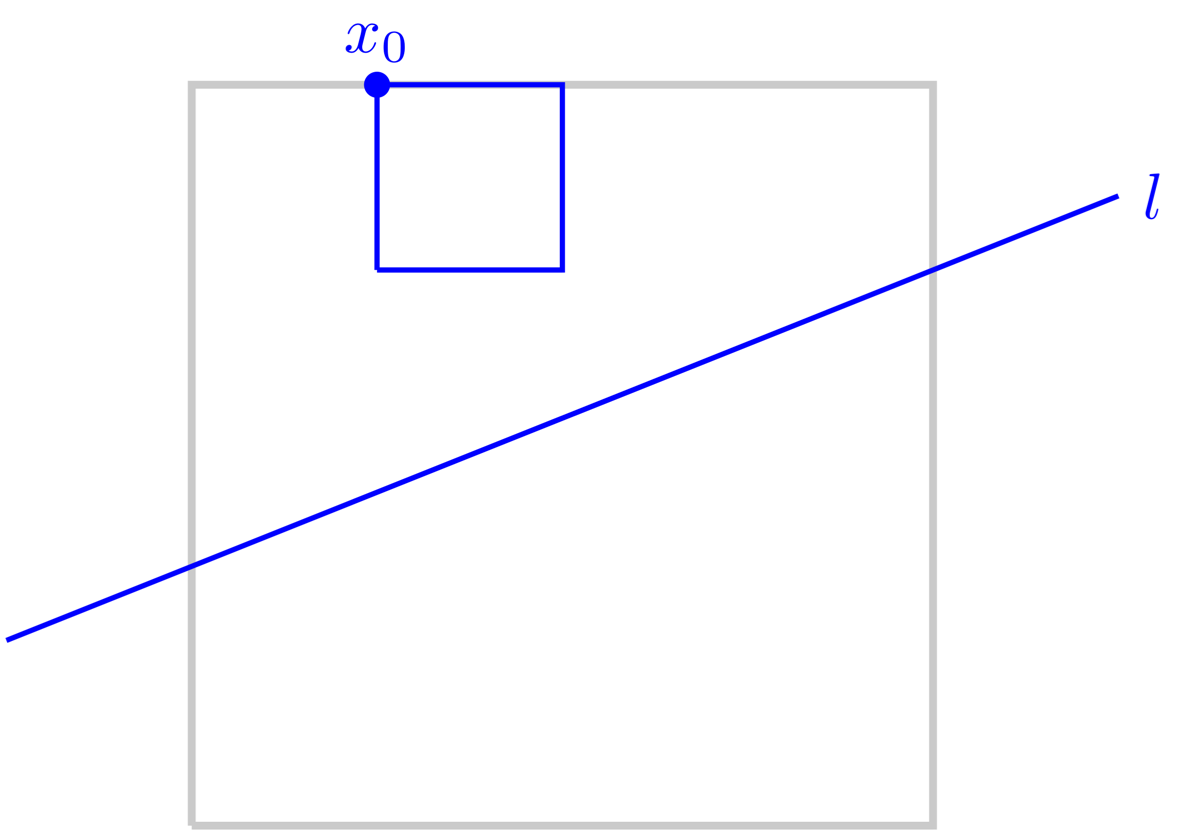}
}\vskip -0.382cm
\caption{The square $f_d\left([0,1]^2\right)$ and the point $x_0$, which may lie in $(0,1)^2$.}\label{corner}
\end{figure}
With no loss of generality, we may require that $l$ misses each of the four vertices of $[0,1]^2$ hence intersects two of the open segments $\{i\}\times(0,1)$ and $(0,1)\times\{i\}$, with $i=0,1$. Pick a component of $\bbR^2\setminus l$ that intersects $K$ and denote the other component by $U$. Let $\delta>0$ be the smallest number such that the closed  $\delta$-neighborhood of  $U$, denoted by $U_\delta$, contains $K_1$. Further pick $d\in\Dc$ such that $f_d\left([0,1]^2\right)$ intersects the boundary $\partial U_\delta$ at a single point $x_0$, which is necessarily a corner of  $f_d\left([0,1]^2\right)$. Here it is possible that $x_0$ lies in the interior of $[0,1]^2$.

Let $l_j=f_d^j(l)$, which is a straight line consisting of all points  $f_d^j(x)$ with $x\in l$. The countable intersection $\bigcap_{j\ge1}f_d^j\left([0,1]^2\right)$ consists of a single point, say $x_\infty$, which is just the unique fixed point of  $\displaystyle f_d(x)=\frac{x+d}{N}$. Let $U_j$ be the  component of $\bbR^2\setminus l_j$ that contains $x_0$. Then $U_j\supset U_{j+1}$ holds for all $j\ge1$. Moreover, the component of $K$ containing $x_\infty$ is a subset of $\bigcap_jU_j$. That is to say,  $\{x_0\}$ is a point component of $K$ and hence $\lambda_K(x_0)=0\in\lambda_K(K)$.
\end{proof}

\section{Proving Theorem \ref{maintheo:criterion}}\label{proof_3}

 Given a fractal square $K=K(N,\Dc)$, by Theorem \ref{maintheo:LLR_improved} we know that $K$ is a Peano compactum if only the complement of $H=K+\bbZ^2$ has no component that is unbounded. By Theorem \ref{maintheo:product_form}, we further know that $\lambda_K(K)=\{1\}$ if and only if $\Dc$ is of product form. In the current section, we assume  $\lambda_K(K)=\{0,1\}$ and characterize $\lambda_K^{-1}(1)$ by analyzing some of its special subset.

Those subsets, to be obtained in the sequel, will help us to  characterize all fractal squares $K$ that satisfy $\dim_H\lambda_K^{-1}(1)>1$ or   
$\dim_H\lambda_K^{-1}(1)=1$ or  $\lambda_K^{-1}(1)=\emptyset$.

Hereafter in this section, we fix $N\ge3$ and assume that $K$ has both points and line segments as components. In such a case, it is obvious that $K\ne[0,1]^2$. Moreover, let $\mathcal{G}_N$ denote the semi-group generated by the translation group $\bbZ^2$  and the $N$-homothety $\psi_N(x)=Nx$.

By Theorem \ref{maintheo:LLR_improved}, the line segments  contained in $K$ are parallel and have the same  rational slope $\tau=\frac{r}{s}\in\bbQ$, with $r\in\bbZ\setminus\{0\}$ and $s\in\bbZ_+$. The invariance of $H=K+\bbZ^2$  under $\psi_N$ implies that $H$ contains at least one line $l_{\omega_0}: x_2=\omega_0+\tau x_1$ and hence all the lines $l_{\omega_0}+\bbZ^2$.

Let $H_n=K^{(n)}+\bbZ^2$. Let  $\Pi(x)=x_2-\tau x_1$ for $x=(x_1,x_2)\in\bbR^2$.
Then for any closed set $X$ invariant under $\Gc_N$ the complement of $\Pi(X^c)$ consists of all $\omega$ such that  $l_\omega: x_2=\omega+\tau x_1$  is contained in $X$. Since $H_n\supset \psi_N\left(H_{n+1}\right)$ holds for all $n\ge1$, we shall have
\begin{equation}\label{eq:H_n^c}
H_n^c \ \subset\ \bbR^2\setminus \psi_N\left(H_{n+1}\right) %=\psi_N\left(\bbR^2\right)\setminus \psi_N\left(H_{n+1}\right)
=\psi_N\left(\bbR^2\setminus H_{n+1}\right)=\psi_N\left(H_{n+1}^c\right).
\end{equation}
Since $\psi_N\circ\prod(x)=\prod\circ\psi_N(x)$ foe all $x\in\mathbb{R}^2$, the containment  $\Pi\left(H_n^c\right)\subset\psi_N\circ\Pi\left(H_{n+1}^c\right)$ holds for all $n\ge1$. From this, it follows that 
\begin{equation}\label{equ:Pi_H^c}
\bbR\setminus\Pi(H_n^c) \ \supset \ \bbR\setminus\psi_N\circ\Pi(H_{n+1}^c)=\psi_N\left(\bbR\setminus\Pi(H_{n+1}^c)\right).
\end{equation}
To facilitate our presentation, we further set 
\begin{align}
\displaystyle \Omega& =\left[0,\frac1s\right]\setminus\Pi(H^c)\\
\displaystyle \Omega_n& =\left[0,\frac1s\right]\setminus\Pi\left(H_n^c\right)\ (n\ge1)
\end{align} 
Then we have some basic observations, which either are immediate or may be verified routinely.
\begin{itemize}
\item[(a)] $\Omega_1\supset\Omega_{2}\supset\cdots$ and $\Omega=\bigcap_j\Omega_j$. Moreover,  $\bbR\setminus\Pi(H^c)=\Omega+\frac1s\bbZ$ and $\bbR\setminus\Pi(H_j^c)=\Omega_j+\frac1s\bbZ$.

\item[(b)]  $\Omega_1$ contains $q=q_\Dc\le N$  isolated points $v_1,\ldots, v_q$ (or none) and $m=m_\Dc\le N$ (or none)  intervals, say $\displaystyle\left[\frac{u_1}{Ns},\frac{u_1+1}{Ns}\right],\ldots,
    \left[\frac{u_m}{Ns},\frac{u_m+1}{Ns}\right]$ with $0\le u_1\le\ldots\le u_m\le N-1$, each of which is a component of $\Omega_1$. 
    As $\omega_0\in\Omega_1$,  we have either $q=q_\Dc\ge1$
    (hence $A_1=\{v_1,\ldots,v_q\}\ne\emptyset$) or $m=m_\Dc\ge1$. 
\item[(c)] For each $v_i$ (if $q\ge1$), $l_{v_i}: x_2=v_i+\tau x_1$ lies in $\bigcap_jH_j$ hence $v_i\in\bigcap_j\left(\Omega_j+\frac{\bbZ}{s}\right)=\Omega+\frac1s\bbZ$.
\item[(d)] $\psi_N\left(\Omega_j\right)$ contains $\bigcup_{u=0}^{s-1}(\Omega_{j+1}+u)$ for all $j\ge1$. Check Equation (\ref{equ:Pi_H^c}). Thus $\Omega_j$ contains $\Phi\left(\Omega_{j+1}\right):=h_{u_1}\left(\Omega_{j+1}\right)\cup\cdots\cup h_{u_m}\left(\Omega_{j+1}\right)$, where $h_{u_i}(t)=\frac{t+u_i}{N}$. Check (b).

\item[(e)] $A_1\cup\Phi\left(\left[0,\frac1s\right]\right)\supset\Omega_1$. Check (b) and (d). Moreover, we have
\begin{equation}\label{eq:Omega_sequence_0}
    A_1\cup\Phi(A_1)\cup\Phi^2\left(\left[0,\frac1s\right]\right)\supset  \Phi(\Omega_1).
\end{equation}
By applying $\Phi$ repeatedly to both sides of (\ref{eq:Omega_sequence_0}), for each $n\ge2$ we shall have
\begin{equation}\label{eq:Omega_sequence_1}    A_1\cup\Phi(A_1)\cup\cdots\cup\Phi^{n-1}(A_1)\cup
    \Phi^n\left(\left[0,\frac1s\right]\right)\supset  \Phi^{n-1}(\Omega_1).
\end{equation}
\end{itemize}
With those observations, we have all the ingredients to obtain the following.
\begin{theo}\label{theo:Omega}
Assume  $m\ge2$ and set $E_n=\bigcup_{i=0}^{n-1}\Phi^i(A_1)$. Then   $\lim\limits_{n\rightarrow\infty}\left(\Phi^n\left(\left[0,\frac1s\right]\right)\cup E_n\right)$ exists and equals $\Omega$, under the Hausdorff distance. Consequently, $\dim_H\Omega=\frac{\log m}{\log N}$.
\end{theo}
\begin{proof} On the one hand,  $\Phi^i(A_1)\subset\Omega$ for all $i\ge1$. Check (c) and (d). Thus $E_n\subset\Omega$ for all $n\ge1$. On the other,  the  limits  $\lim\limits_{n\rightarrow\infty}\Phi^{n}(A_1)$ and $\lim\limits_{n\rightarrow\infty}\Phi^{n}\left(\left[0,\frac1s\right]\right)$  under the Hausdorff distance 
%(between nonempty compact subsets of $\bbR$) 
both exist and coincide with the self-similar set $E_\infty$ determined by the IFS $\{h_{u_i}: 1\le i\le m\}$. Therefore, each  $\omega\in E_\infty$ may be approached by a sequence $\{\omega_n: n\ge1\}$ with $\omega_n\in\Phi^n(A_1)$. Every $\omega_n$ corresponds to an infinite line $l_{\omega_n}: x_2=\omega_n+\tau x_1$ that is contained in $\Omega$. This forces that $\omega\in\Omega$ and $\Omega=\lim\limits_{n\rightarrow\infty}\left(\Phi^n\left(\left[0,\frac1s\right]\right)\cup E_n\right)$. The second part is immediate, so we are done.
\end{proof}

\begin{rema}
The major ideas of Theorem \ref{theo:Omega} were already used in \cite[Theorem 1.2]{ZhangYF-20}. We summarize the fundamentals here and illustrate the logic behind the arguments
by using Barnsley's IFS theory of attractor with condensation \cite{Barnsley93}. 
\end{rema}

\begin{rema}
Besides the special case $m\ge2$ considered in Theorem \ref{theo:Omega}, there are other cases. For instance, if $m=1$ and $q\ge1$ then $\Omega=\lim\limits_{n\rightarrow\infty}E_n$, where $E_n=\bigcup_{i=0}^{n-1}\Phi^i(A_1)$. This set consists exactly of $\bigcup_nE_n$  and the only fixed point $\omega_1$ of $h_{u_1}:\bbR\rightarrow\bbR$. However, if $m\le1$ and $mq=0$ there are two sub-cases, either (1) $m=1$ and $q=0$ or (2) $m=0$ and $q\ge1$. In the first, $\Omega=\Omega_1=\{\omega_1\}$. In the second, $\Omega=\Omega_1=A_1$.
\end{rema}

Let  $\Hc=\Hc_K$ consist of all the infinite lines in $H$. With Theorem \ref{theo:Omega}, we have the following.

\begin{theo}\label{theo:intercept_set} Each of the following assertions hold:
\begin{itemize}
\item[\rm(i)] $\dim_H\lambda_K^{-1}(1)=\dim_H(\Hc\cap K)$.
\item[\rm(ii)] If $m\ge2$ then  $\dim_H\lambda_K^{-1}(1)=1+\frac{\log m}{\log N}$. 
\item[\rm(iii)] If $m=1$ and $q\ge1$, then $\dim_H\lambda_K^{-1}(1)=1$. 
\item[\rm(iv)] If  $m\le1$ and $mq=0$ then $\Omega$ is a finite set and $\lambda_K^{-1}(1)=\emptyset$. 
\end{itemize}
\end{theo}

\begin{proof}%[{\bf Proof for Theorem  \ref{theo:intercept_set}}]
 Recall that $\mathcal{G}_N$ is a countable semi-group and that $\Hc$ is invariant under $\mathcal{G}_N$. For any point  $x\in\lambda_K^{-1}(1)$, the component of $K$ containing $x$ is a line segment, say $L_x$; moreover, there is an infinite sequence of line segments $L_j$ whose limit under the Hausdorff distance is a non-degenerate lien segment containing $x$. Thus we can find $g\in\mathcal{G}_N$ such that $g(x)\in\Hc$. This proves item (i). The other items are proved  below.
 
 If $m\ge2$, we can apply Theorem \ref{theo:Omega} to infer that  $\displaystyle\dim_H\Omega=\frac{\log m}{\log N}$. This proves item (ii).

If  $m=1$ and $q\ge1$, then $A_1=\{v_1,\ldots.v_q\}$ contains exactly $q$ points. Let $E_n=\bigcup_{i=0}^{n-1}\Phi^i(A_1)$. Then it is routine to show that $\Omega=\lim\limits_{n\rightarrow\infty}E_n$ consists of $\bigcup_nE_n$  and the only fixed point $\omega_1$ of $h_{u_1}:\bbR\rightarrow\bbR$. This implies that $\Omega$ is a countable set and  $\Hc$ contains at least one infinite sequence of lines, with intercepts $h_{u_1}^n(v_1)$, that converge to the infinite line $l_{\omega_1}: x_2=\omega_1+\tau x_1$.  The equality $\dim_H\lambda_K^{-1}(1)=\dim_HK_c=1$ then follows. This proves item (iii).

Finally, if $m\le1$ and $mq=0$, then $\Omega=A_1$, as given in Observation (b). This then implies that $\Omega+\frac{\bbZ}{s}$ is a uniformly discrete set hence $H$ does not contain  any infinite sequence of straight lines converging to a limit line under Hausdorff distance. This proves item (iv).
\end{proof}

To conclude this section, we give the following.

\begin{exam}\label{ex:(2,1)} 
Let  $K=K(N,\Dc)$ be a fractal square of order $N=7$. See  Figure \ref{fig:slope_1_dim} for the first approximation $K^{(1)}$ (in the left), the second approximation $K^{(2)}$ (in the middle), and the third approximation (in the right). Moreover, we shall have $(m_\Dc,q_\Dc)=(2,1)$.
\begin{figure}[ht]
\begin{center}
\begin{tabular}{ccc}
\includegraphics[width=5.28cm]{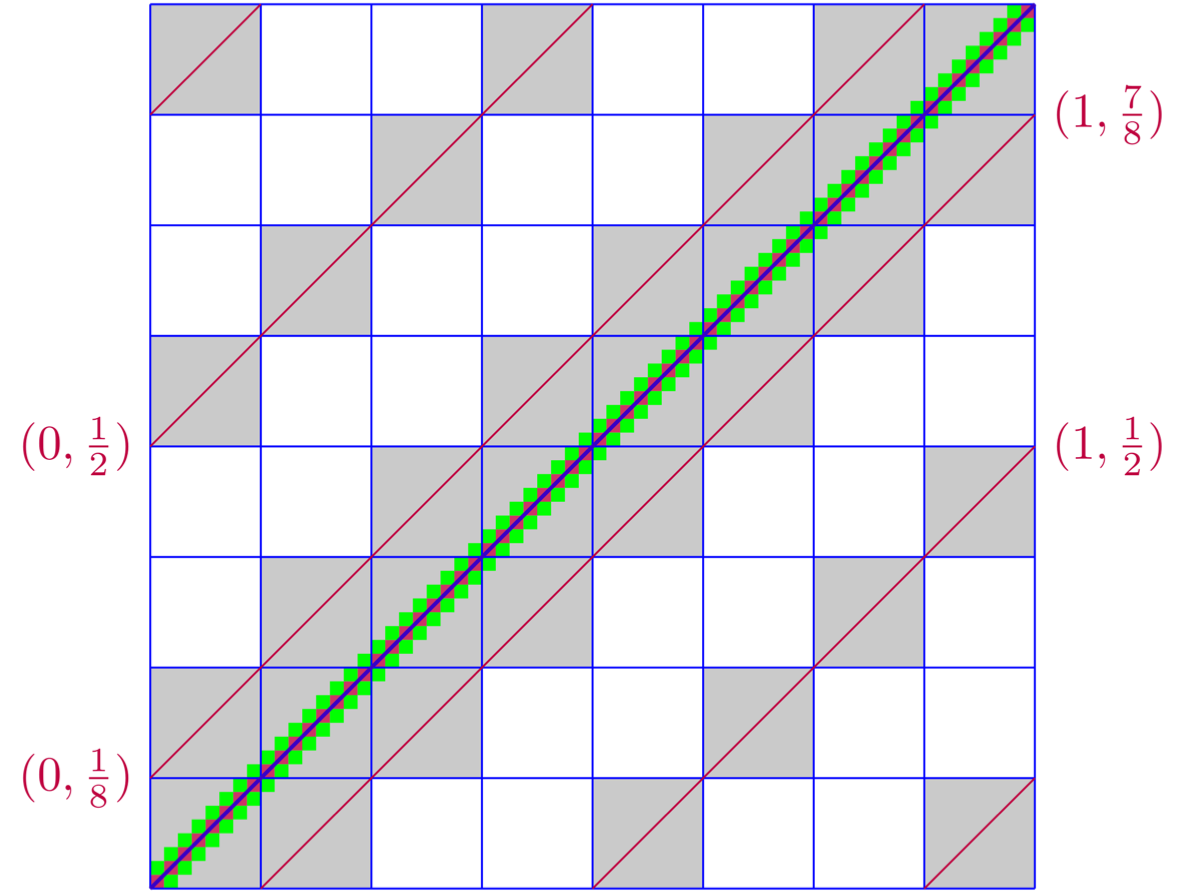}
&\includegraphics[width=4.0cm]{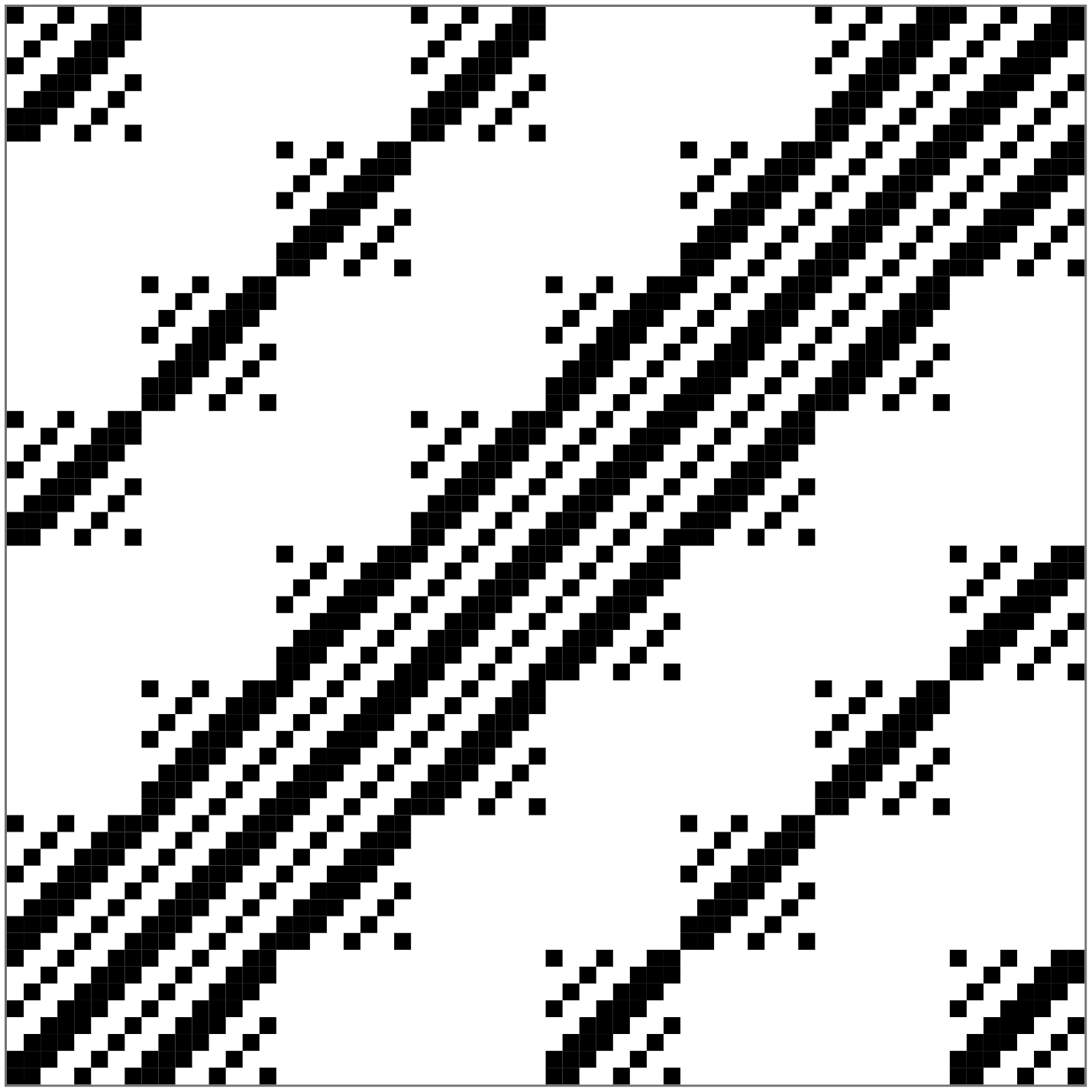}
&\includegraphics[width=4.0cm]{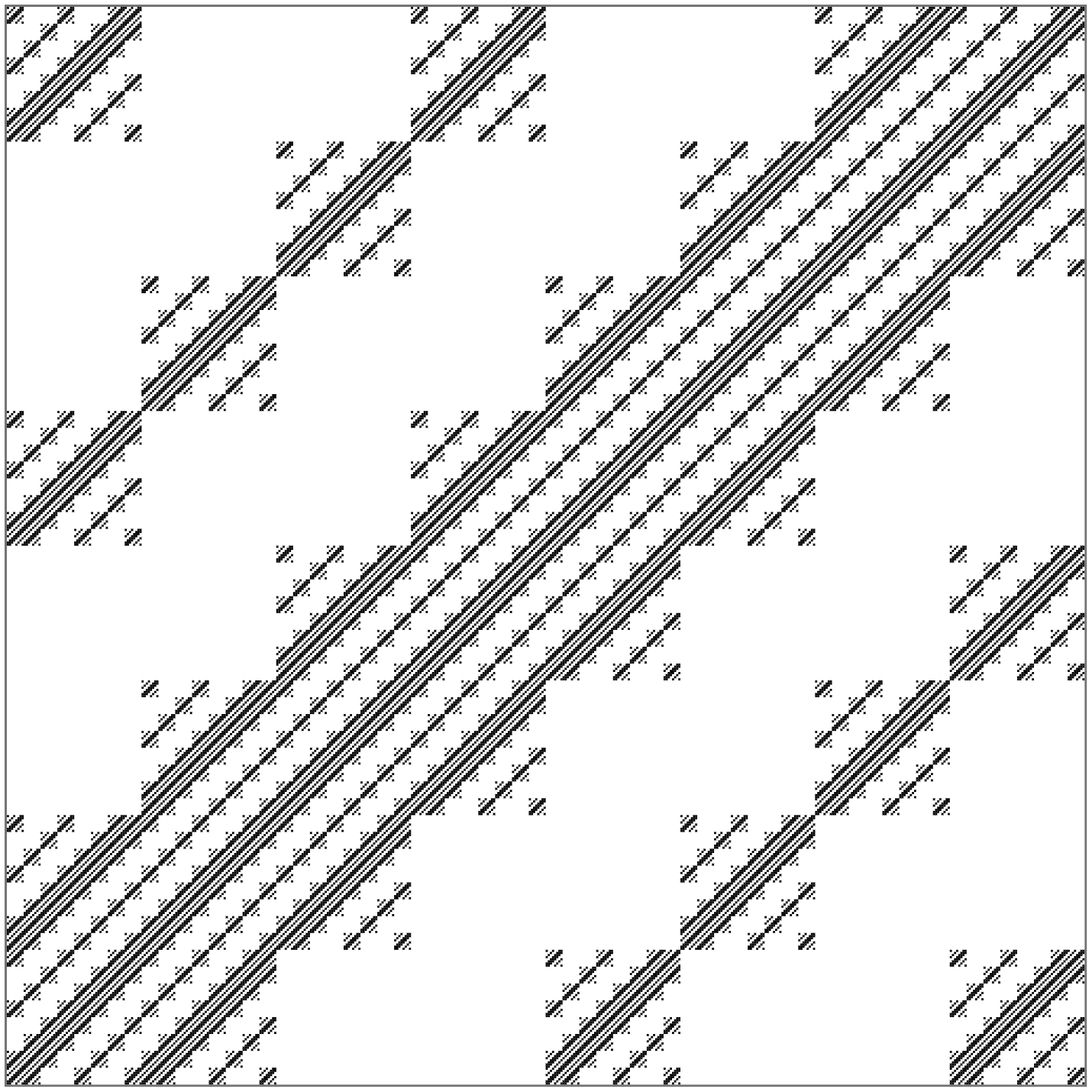}
\end{tabular}
\end{center}
\vspace{-0.5cm}
\caption{A fractal square $K(N,\Dc)$ with $(m,p)=(2,1)$.}\label{fig:slope_1_dim}
\end{figure}
\end{exam}

%To conclude this section, we give a couple of basic observations, which either are immediate or can be checked routinely. First, if no component of $K^{(1)}$ intersects each of the four sides $\alpha_i=\{i\}\times[0,1], \beta_i=[0,1]\times\{i\}$ with $i=0,1$ then one can find either a Jordan arc $\gamma$ in $[0,1]^2\setminus K^{(1)}$ that intersects both $\{0\}\times(0,1)$ and $\{1\}\times(0,1)$ or one that intersects both $(0,1)\times\{0\}$ and $(0,1)\times\{0\}$. Second, if $K^{(1)}$ has a component disjoint from one of $\alpha_i\cup\beta_j(i,j=0,1)$ then $K$ has a point component and $\beta_0\left(K^{(1)}\right)>\beta_0\left(K^{(2)}\right)$. Third, if $\beta_0\left(K^{(1)}\right)=\beta_0\left(K^{(2)}\right)>1$ then there are two possibilities: (1) every component of $K^{(1)}$ intersects both $(0,1)\times\{0\}$ and $(0,1)\times\{0\}$, (2)  every component of $K^{(1)}$ intersects both $\{0\}\times(0,1)$ and $\{1\}\times(0,1)$. Those digit sets are referred to as {\bf NS digit sets} and {\bf EW digit sets}, respectively. Special NS and EW digit sets are also given in Definition \ref{def:product_form}.

\section{There are Planar Compacta $L_n(n\ge1)$ with $\lambda_{L_n}(L_n)=\{0,\ldots,n\}$}\label{examples_and_questions}

We construct a sequence of compacta $L_n\subset\bbR^2(n\ge2)$ with two requirements. First,  all the components of $L_n$ are locally connected. Second,  $\lambda_{L_n}(L_n)=\{0,\ldots,n\}$.
We will need  the Cantor ternary set $\mathcal{C}$ and  the components of $[0,1]\setminus\mathcal{C}$, to be denoted by $\{(a_j,b_j): j\ge1\}$. Assume that $a_1=\frac13$ and $b_1=\frac23$. Let $Q$ consist of the boundary of $[0,1]^2$ and the line segments 
$l_k'=\left\{\left(\frac1k,t\right):\  \frac1k\le t\le 1-\frac1k\right\}$ and $l_k''=\left\{\left(1-\frac1k,t\right):\ \frac1k\le t\le 1-\frac1k\right\}$ 
with $k\ge3$. See Figure \ref{key_picture}(a).
\begin{figure}[ht]
\vspace{-0.15cm}
\begin{center}\begin{tabular}{ccc}
\includegraphics[width=4.9cm]{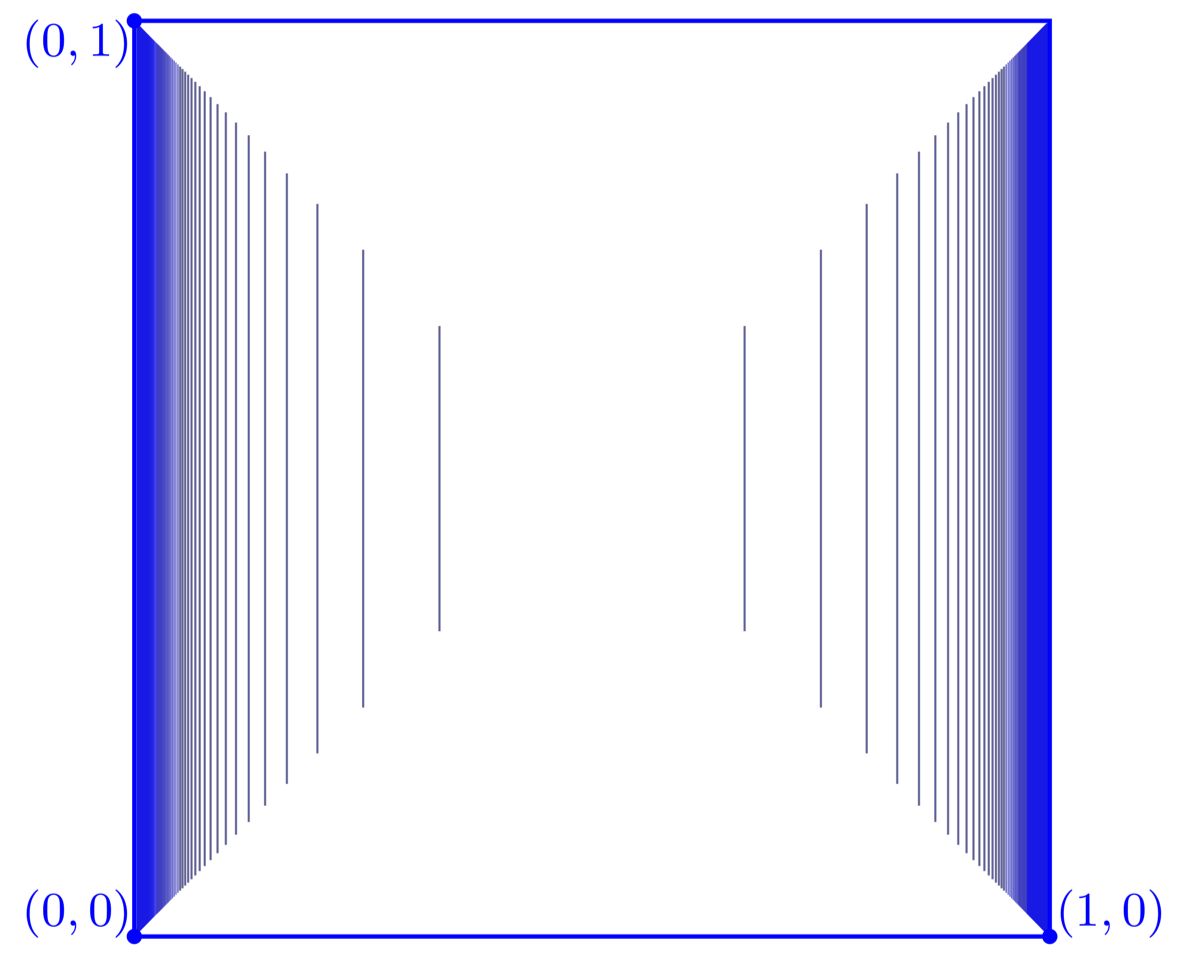} & \includegraphics[width=4.65cm]{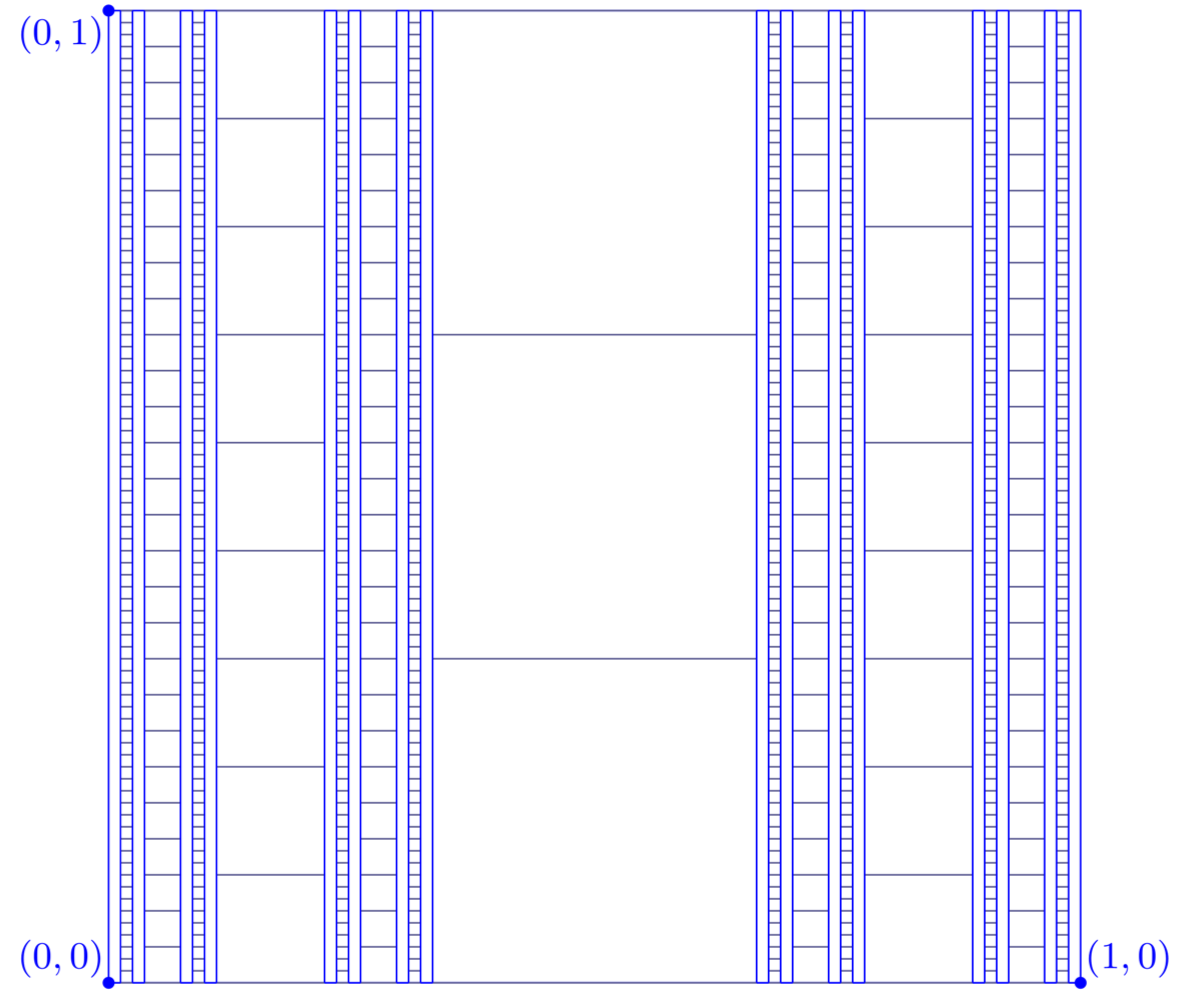} &
 \includegraphics[width=2.5cm]{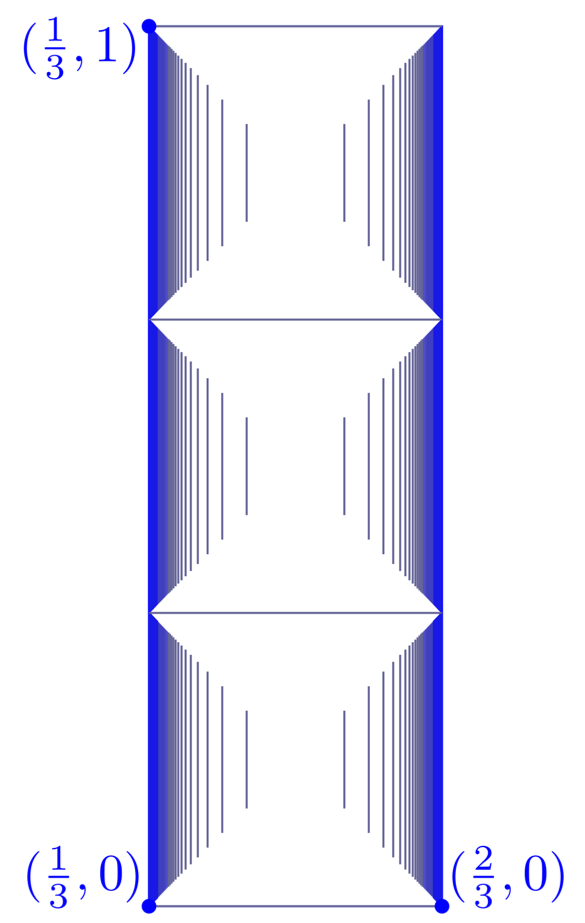}
 \\ (a) & (b) &(c) \end{tabular}
\end{center}
\vspace{-0.5cm}
\caption{(a) the continuum $Q$; (b) the continuum $L_0$; (c) $L_1\bigcap\Big([\frac13,\frac23]\times[0,1]\Big)$.}\label{key_picture}
\end{figure}
Now, for all $j\ge1$ we pick $k_j$ such that $b_j-a_j=3^{-k_j}$ and let  $M_j$ consist of   
all the segments $\displaystyle l_{i,j}=\left\{\left(t,i\cdot3^{-k_j}\right): a_j\le t\le b_j\right\}
$ with $0\le i\le 3^{k_j}$. See  Figure \ref{key_picture}(b). Then, the union of $\mathcal{C}\times[0,1]$ and $\bigcup_{j\ge1}M_j$, to be denoted by $L_0$, is a locally connected continuum. 
For $j\ge1$ and $0\le i\le 3^{k_j}$, set $\displaystyle f_{i,j}(x)=3^{-k_j}x+\left(a_j,i\cdot3^{-k_j}\right)$ and
\[ \Lambda_0=\bigcup_{j\ge1}\bigcup_{i=0}^{3^{k_j}-1}f_{i,j}(Q).
\]
Further set
$\Lambda=\alpha_\infty\cup\left(\bigcup_{p}\alpha_p\right)$ 
where $\alpha_\infty$ is the line segment connecting $(0,1)$ to $(1,1)$ and  $\alpha_p(p\ge1)$ the one connecting $\left(\frac{1}{p+2},1+\frac{1}{p+2}\right)$ to  $\left(\frac{p+1}{p+2},1+\frac{1}{p+2}\right)$. 

\begin{exam}\label{key_example_1}
$L_1=L_0\cup\Lambda_0$ has locally connected components and satisfies $\lambda_{L_1}(L_1)=\{0,1\}$, since the non-degenerate atoms of $L_1$ are exactly the line segments $\{u\}\times[0,1]$ with $u\in\mathcal{C}$. See Figure \ref{key_picture}(c) for a simple depiction of $L_1\bigcap\Big([\frac13,\frac23]\times[0,1]\Big)$.
\end{exam}

\begin{exam}\label{key_example_2}
$L_2=\Lambda\cup L_1$ has locally connected components and satisfies $\lambda_{L_2}(L_2)=\{0,1,2\}$. 
In deed, $L_2$ has just one non-degenerate atom, say $\delta_2$,  which consists $\alpha_\infty$, $\mathcal{C}\times[0,1]$ and 
$[1,2]\times\mathcal{C}$.
The left part of  Figure \ref{Fig13} illustrates $\Lambda\cup\delta_2$. Clearly,  $\lambda_{\delta_2}(\delta_2)=\{0,1\}$ hence $\lambda_{L_2}(L_2)=\{0,1,2\}$. 
\end{exam}
\begin{figure}[ht]
\vspace{-0.382cm}
\begin{center}
\begin{tabular}{ccc}
\includegraphics[width=4.0cm]{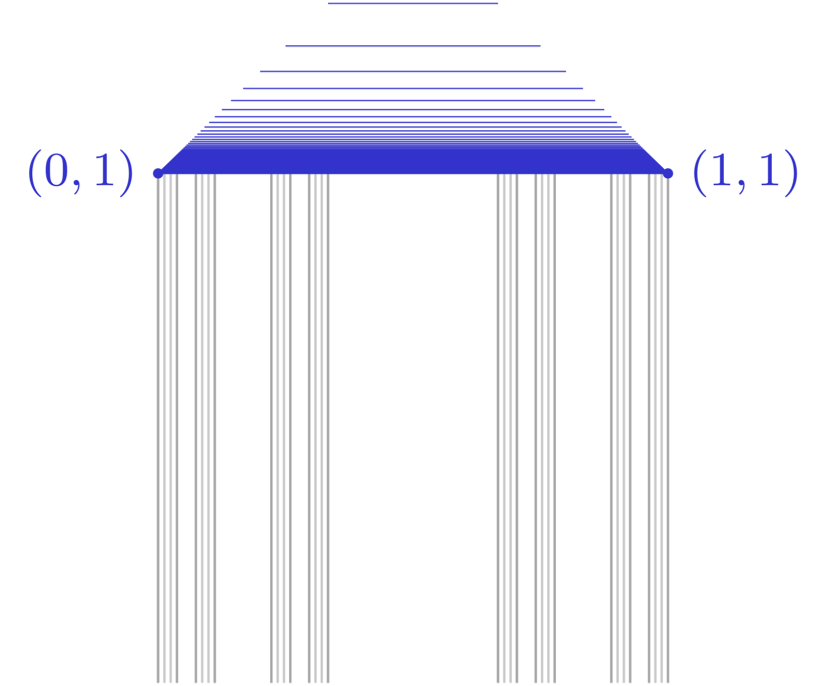}
&
\includegraphics[width=5.0cm]{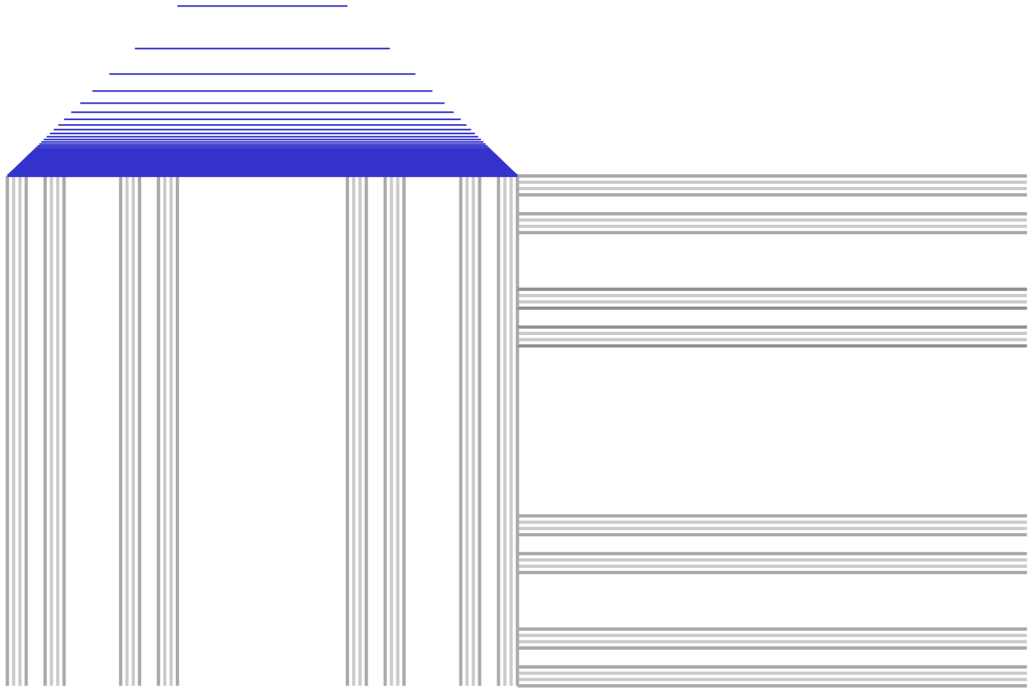}
&
\includegraphics[width=5.0cm]{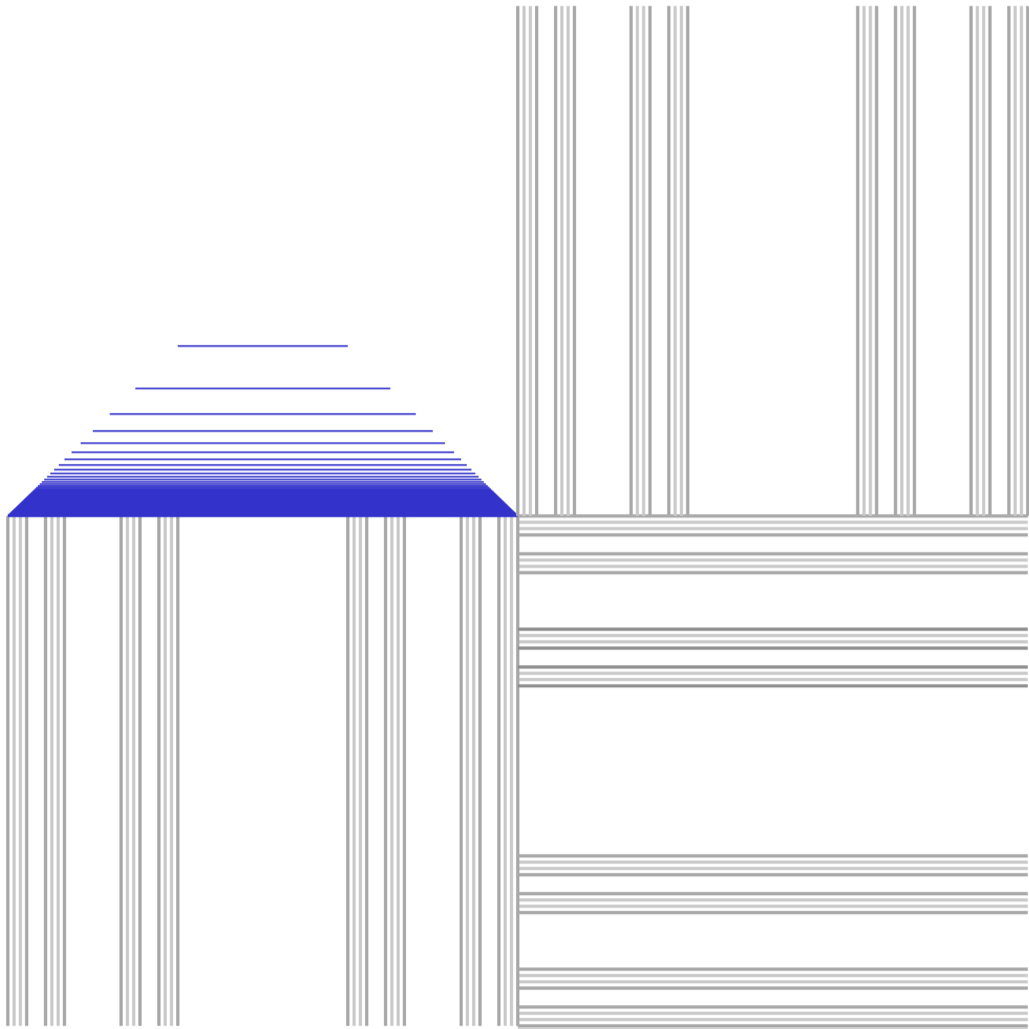}
\\
(a) & (b) &(c)
\end{tabular}
\end{center}
\vspace{-0.382cm}
\caption{(a)  $\Lambda\cup\delta_2$; \quad  (b) $\Lambda \cup\delta_3$; \quad (c) $\Lambda\cup\delta_4$.}\label{Fig13}
\end{figure}

Notice that $L_2=L_1\cup L_1'$, where  $L_1'=\{(u,v): (v,u)\in L_1\}+(1,0)$. Now we can inductively define $L_n(n\ge3)$ in the following way.
If $n=2k$ for  $k\ge1$, let $L_n=L_{2k-1}\cup L_1'+(k-1,k-1)$. If $n=2k+1$ for $k\ge1$, let $L_n=L_{2k}\cup L_1+(k-1,k-1)$.

\begin{exam}\label{key_example}
$L_n(n\ge3)$ has locally connected components and satisfies $\lambda_{L_n}(L_n)=\{0,\ldots,n\}$. The are four basic observations. First,  
$\delta_n\subset\delta_{n+1}$ fro all $n\ge2$. Second, each component of $L_n$ is locally connected. Third, $L_n$ has just one non-degenerate  atom $\delta_n$. The middle and the right parts of Figure \ref{Fig13} illustrate $\Lambda\cup\delta_3$ and $\Lambda\cup\delta_4$.
 Finally,  every atom of $\delta_{n+1}$ is either a point $(u,1)$ with $u\notin\mathcal{C}$, a segment $\{u\}\times[0,1]$ with $u\in\mathcal{C}$, or a continuum congruent to $\delta_n$. A similar sequence has been analyzed in \cite[Example 4.6]{LYY-2019}, with motivations to compute the NLC-scale of certain planar continua. Clearly, $\lambda_{\delta_n}(\delta_n)=\{0,\ldots,n-1\}$ and hence $\lambda_{L_n}(L_n)=\{0,\ldots,n\}$.
\end{exam}

%\cite{ZhangYF-20}

\noindent
{\bf Acknowledgment}. The current study was funded by the National Key R\&D Program of China [No. 2024YFA 1013700] and  the Guangdong Basic and Applied Basic Research Foundation (Grant No. 2021A1515010242). The third named author is grateful to the hospitality of Brigham Young University, for supporting his two visits  in 2024 and 2025. The authors also want to thank Hui RAO and Yuan ZHANG at Huazhong Normal University in Wuhan and Huojun RUAN at Zhejiang University in Hangzhou, for helpful suggestions. All pictures are in png form. The colored ones are drawn in tikz packages. The others are made by a new program, which is recently composed by Greg Conner at Brigham Young University in 2025. 

%%% bibliography
\bibliographystyle{plain}
\bibliography{biblio-2}

\end{CJK}
\end{document}